\documentclass[10pt]{article}

\usepackage[paper=a4paper,left=29mm,right=29mm,top=28mm, bottom=28mm]{geometry}

\usepackage{hyperref}

\usepackage{academicons}

\usepackage{xcolor}

\usepackage[version=4]{mhchem}

\usepackage[ruled,vlined]{algorithm2e}

\usepackage{tikz,tikz-qtree}
\usetikzlibrary{positioning,arrows,arrows,automata,shapes}
\tikzset{
  invisible/.style={opacity=0},
  visible on/.style={alt={#1{}{invisible}}},
  alt/.code args={<#1>#2#3}{%
    \alt<#1>{\pgfkeysalso{#2}}{\pgfkeysalso{#3}} 
  },
}

\usepackage{authblk}
\usepackage{tikz}
\usepackage{xfrac}
\usepackage{tabularx, multirow}
\usepackage{graphicx, tikz}
\usetikzlibrary{arrows.meta}
\usetikzlibrary{arrows,calc}
\usetikzlibrary{shapes}

\usepackage{latexsym,amssymb,amsfonts,graphicx,makeidx,paralist}
\usepackage{theorem,rotating,ifthen,amsmath,subfigure,epsfig,nicefrac}
\usepackage{booktabs,framed}

\usepackage{color,framed, float,wasysym,cancel}

\colorlet{shadecolor}{gray!12}

\newcommand{\NN}{{\mathbb N}}
\newcommand{\gr} {\mbox{Digraph}}

\newcommand{\g} {\mbox{digraph}}
\newcommand{\SPD} {\text{SPD}}
\newcommand{\MSP} {\text{MSP}}

\newenvironment{desctight}
  {\begin{list}{}{\setlength\labelwidth{0pt}
        \setlength{\itemsep}{0.5pt}
        \setlength{\parsep}{0pt}
        \setlength\itemindent{-\leftmargin}
        }}
    {\end{list}}

\newtheorem{theorem}{Theorem}[section]
\newtheorem{example}[theorem]{Example}
\newtheorem{corollary}[theorem]{Corollary}
\newtheorem{lemma}[theorem]{Lemma}
\newtheorem{observation}[theorem]{Observation}

\newtheorem{definition}[theorem]{Definition}

\newtheorem{remark}[theorem]{Remark}
\newtheorem{proposition}[theorem]{Proposition}

\newenvironment{proof}{\noindent{\bf Proof~}}{\null\hfill $\Box$\par\medskip}

\newcommand{\bigo}{\text{$\mathcal O$}}

\newcommand{\MSOA}{\text{MSO}_1}

\newcommand{\ideg}{\text{indegree}}
\newcommand{\odeg}{\text{outdegree}}
\newcommand{\un} {{\it un}}

\newcommand{\LMSOA}{\text{LinEMSO}_1}

\newcommand{\dcws} {\text{d-cw}}

\newcommand{\IN}{\mathbb{N}}
\newcommand{\fpt} {\mbox{FPT}}
\newcommand{\xp} {\mbox{XP}}
\newcommand{\w} {\mbox{W}}

\newcommand{\free} {\text{Free}}

\newcommand{\OC} {\text{OC}}

\newcommand{\SPO} {\text{SPO}}

\newcommand{\CN} {\text{CN}}

\newcommand{\OCN} {\text{OCN}}

\newcommand{\dmws} {\text{dmw}}
\newcommand{\arc} {\mbox{arc}}

\newcommand{\np} {\mbox{NP}}
\newcommand{\p} {\mbox{P}}

\begin{document}


\title{Efficient computation of the oriented chromatic number of recursively defined digraphs\thanks{An extended abstract of
this paper appeared in Proceedings of the {\em International Conference on Combinatorial Optimization and Applications} (COCOA 2020) \cite{GKL20}.}}

\author[1]{Frank Gurski}
\author[1]{Dominique Komander}
\author[1]{Marvin Lindemann}

\affil[1]{\small University of  D\"usseldorf,
Institute of Computer Science, Algorithmics for Hard Problems Group,
40225 D\"usseldorf, Germany}

\maketitle


\begin{abstract}
In this paper we consider colorings of oriented graphs, i.e.\ digraphs
without cycles of length 2. Given some  oriented graph  $G=(V,E)$,  an oriented $r$-coloring for $G$ is a partition 
of the vertex set $V$ into $r$ independent sets, such
that all the arcs between two of these sets have the same direction.
The oriented chromatic number of  $G$ is the  smallest 
integer $r$ such that $G$
permits an oriented $r$-coloring.
Deciding whether an acyclic digraph  has an oriented $4$-coloring
is NP-hard, which motivates to consider the problem
on special graph classes.

In this paper we consider the Oriented Chromatic Number problem on classes  of recursively defined 
oriented graphs.
Oriented co-graphs (short for  oriented complement reducible graphs) can 
be recursively defined defined from
the single vertex graph by applying the disjoint union and order composition.
This recursive structure allows to compute an optimal oriented coloring and the
oriented chromatic number in linear time. 
We generalize this result using the  concept of perfect orderable graphs. 
Therefore, we show that for acyclic transitive digraphs every greedy coloring along
a topological ordering leads to an optimal oriented coloring.
Msp-digraphs (short for minimal series-parallel digraphs) can be defined from
the single vertex graph by applying the  parallel composition and series composition.
We prove an upper bound of $7$ for the oriented chromatic
number for msp-digraphs and we give an example to show
that this is bound best possible. We apply this bound and 
the recursive  structure of
msp-digraphs to obtain a linear time solution for
computing the oriented chromatic number of msp-digraphs.

In order to generalize the results on computing the oriented chromatic number
of special graph classes, we consider the parameterized complexity of the Oriented Chromatic Number problem
by so-called structural parameters, which
are measuring the difficulty of decomposing a graph into a special tree-structure.

\bigskip
\noindent
{\bf Keywords:} oriented graphs;
msp-digraphs;
oriented co-graphs;  
oriented coloring; 
efficient algorithms; 
directed clique-width; 
parameterized algorithms

\end{abstract}


\section{Introduction}

Given some undirected graph
$G=(V,E)$, an {\em $r$-coloring} for $G$   is a partition of the vertex set $V$ into $r$ independent sets. 
The smallest integer $r$ such that a graph
$G$ permits an $r$-coloring is referred to as the
{\em chromatic number} of $G$.
Deciding whether a graph has a $3$-coloring is NP-complete. 
However, there are efficient solutions for the Chromatic Number 
problem on special graph classes, such as  co-graphs \cite{CLS81},
chordal graphs \cite{Gol80}, and comparability graphs \cite{Hoa94}.

For directed graphs the concept of acyclic colorings introduced by Neumann-Lara \cite{NL82}
received a lot of attention in \cite{NL82,Moh03,BFJKM04}
and also in recent works \cite{LM17,MSW19,SW20,GKR20c,GKR21}. 
Given some directed graph  $G=(V,E)$, an {\em acyclic $r$-coloring} for $G$ is a partition of the vertex set $V$ into 
$r$ acyclic sets.\footnote{A set $V'$ of vertices of a digraph $G$ is called {\em acyclic} if 
the subdigraph induced by $V'$ is acyclic.} The {\em dichromatic
number} of a directed graph $G$ is the smallest  integer $r$ such that $G$
permits an acyclic $r$-coloring.

In this paper we consider the principle of oriented colorings on oriented graphs, 
which has been introduced  by Courcelle \cite{Cou94}. 
Given some  oriented graph  $G=(V,E)$,  an {\em oriented $r$-coloring} for $G$ is a partition 
of the vertex set $V$ into $r$ independent sets, such
that all the arcs linking two of these subsets have the same direction.
The {\em oriented chromatic number} of an oriented graph $G$, denoted by $\chi_o(G)$, is the  smallest 
integer $r$ such that $G$
has an oriented $r$-coloring. Oriented colorings have applications in scheduling
models in which incompatibilities are oriented \cite{CD06}.

In the Oriented Chromatic Number problem ($\OCN$ for short) there is given an 
oriented graph $G$ and an integer $r$ and one has to
decide whether there is an  oriented $r$-coloring for $G$. 
If $r$ is constant, i.e.\ not part of the input, the corresponding problem
is denoted by $\OCN_{r}$. Even $\OCN_{4}$ is NP-complete~\cite{CD06}.

So far, the definition of oriented coloring is mostly applied to undirected graphs.
In this case, the maximum value $\chi_o(G')$ of all possible orientations $G'$ of an undirected graph
$G$ is considered.
For several special undirected graph classes the oriented chromatic number has been bounded.
Among these are outerplanar graphs \cite{Sop97}, planar graphs \cite{Mar13}, and
Halin graphs \cite{DS14}. See \cite{Sop16} for a survey.

Oriented colorings of special classes of oriented graphs seems to be nearly uninvestigated.
In this paper, we consider the Oriented Chromatic Number problem restricted to  acyclic transitive digraphs,
oriented co-graphs, msp-digraphs, and  digraphs of bounded directed clique-width.

Oriented complement reducible graphs, oriented co-graphs for short, have been 
studied  by  Lawler in \cite{Law76} and by Corneil et al.\ in \cite{CLS81} using the notation
of transitive series parallel (TSP) digraphs. Oriented co-graphs
can be defined from the single vertex graph by applying the disjoint union and the 
order composition. This recursive structure allows to compute an optimal oriented coloring and the
oriented chromatic number in linear time \cite{GKR19d}. 
We generalize this result using the  concept of perfect orderable graphs by showing that for
acyclic transitive digraphs every greedy coloring along
a topological ordering leads to an optimal oriented coloring.
In order to obtain an upper bound we show that for  acyclic transitive digraphs 
and thus also for  oriented co-graphs  the oriented chromatic
number is at most the maximum vertex degree plus one.

Minimal series-parallel digraphs, msp-digraphs for short, have been  analyzed in \cite{VTL82}.
By \cite[Section 11.1]{BG18} these digraphs can be used for modeling flow
diagrams and dependency charts and have applications within scheduling under constraints.
Msp-digraphs can be defined from
the single vertex graph by using the parallel composition and series composition.
We prove an upper bound of $7$ for the oriented chromatic
number for msp-digraphs and we give an example to verify
that this is bound best possible. We use this bound and 
the recursive  structure of
msp-digraphs to obtain a linear time solution for
computing the oriented chromatic number of msp-digraphs.
Further, we show an upper bound of $3$ for the chromatic number
of underlying undirected graphs of msp-digraphs.

We also consider the parameterized complexity of the Oriented Chromatic
Number problem  parameterized by so-called  structural parameters, which
are measuring the difficulty of decomposing a graph into a special tree-structure.
The existence of an $\fpt$-algorithm\footnote{FPT is the class of all parameterized problems which can be 
solved by algorithms that are exponential only in the size of a fixed parameter while 
polynomial in the size of the input size \cite{DF13}.} or an $\xp$-algorithm\footnote{XP is the class of all 
parameterized problems which can be solved by algorithms that are polynomial  if the 
parameter is considered as a constant \cite{DF13}.} w.r.t.\ some structural parameter
allows an efficient computation of the Oriented Chromatic
Number problem on graph classes of bounded parameter values.

As already mentioned in \cite{GHKLOR14}, the Oriented Chromatic
Number problem is not in $\xp$ when parameterized by directed 
tree-width, directed path-width, Kelly-width, or DAG-width, unless  $\p=\np$.
Better results can be achieved considering the parameter directed clique-width.
By extending our solution on msp-digraphs we can show an
algorithm for the  Oriented Chromatic
Number problem on digraphs on $n$ vertices given by a directed
clique-width $k$-expression with running time in  $\bigo(n\cdot k^2 \cdot  2^{r(r+k)})$.
This implies that the Oriented Chromatic Number problem 
is in $\fpt$ when parameterized by the directed clique-width and $r$, which 
was already known by defineability in monadic second order logic (MSO) \cite{GHKLOR14}.
Thus, for  every integer $r$ it holds that  $\OCN_{r}$
is in $\fpt$ when parameterized by directed clique-width and 
for every class of graphs of bounded  directed clique-width and every 
integer $r$ the $r$-Oriented Chromatic Number problem can be solved in polynomial time.

Ganian has shown  an FPT-algorithm for $\OCN$ w.r.t.\ the parameter
tree-width (of the underlying undirected graph) \cite{Gan09}.
Further, he has shown that $\OCN$ is DET-hard\footnote{DET is the class of decision
problems which are reducible in logarithmic space to
the problem of computing the determinant of an integer valued $n\times n$-matrix.} for classes of
oriented graphs such that the underlying undirected class has bounded rank-width.

Beside these, we consider the standard parameter, i.e.\ the threshold value given in the instance,
and the parameter ''number of vertices''.
In Table \ref{fpt-sum} we summarize the known results for $\OCN$ and $\OCN_r$ parameterized by  
parameters.

 \begin{table}[ht]
 {\small
\begin{center}
\begin{tabular}{l||ll|ll|}
                           & \multicolumn{2}{c|}{$\OCN$} & \multicolumn{2}{c|}{$\OCN_r$} \\
                           \hline
directed tree-width        & $\not\in\xp$   &  Corollary \ref{cor-xp-ro}&$\not\in\xp$  &   Corollary \ref{cor-xp-ro} \\
\hline
directed path-width        &  $\not\in\xp$  &  Corollary \ref{cor-xp-ro}& $\not\in\xp$ & Corollary \ref{cor-xp-ro}  \\
\hline
DAG-width                  & $\not\in\xp$   & Corollary \ref{cor-xp-ro} & $\not\in\xp$ &  Corollary \ref{cor-xp-ro} \\
\hline
Kelly-width                & $\not\in\xp$  &Corollary \ref{cor-xp-ro} & $\not\in\xp$ & Corollary \ref{cor-xp-ro}  \\
\hline
tree-width of $\un(G)$     &  \fpt &\cite{Gan09} & \fpt   & \cite{Gan09} \\
\hline
rank-width of $\un(G)$     &  DET-h     &\cite{Gan09} &  ? &  \\
\hline
directed modular-width     &     ? & & \fpt &  Corollary \ref{ocnk2}    \\
\hline
directed clique-width      &     ? & & \fpt & Corollary \ref{cor11ax}  \\
\hline
standard parameter  $r$    & $\not\in\xp$     & Corollary \ref{cor-xp-r} &  ///   &   \\
\hline
directed clique-width + $r$& \fpt      & Corollary \ref{o}  &       ///  &   \\
\hline
number of vertices $n$     &  \fpt     & Corollary \ref{xp-n} &  \fpt    & Corollary \ref{xp-n}  \\
\hline
\end{tabular}
\end{center}
\caption{Complexity of $\OCN$ and $\OCN_r$ parameterized by  parameters. 
The ''///'' entries indicate that by taking $r$ out of the instance the considered parameter
makes no sense and the ''?'' entries indicate that the parameterized complexity remain open.  We assume that $\p\neq \np$.}\label{fpt-sum}
}
\end{table}

\section{Preliminaries}\label{intro}

We use notations of Bang-Jensen and Gutin \cite{BG09} for graphs and digraphs.

\subsection{Graphs}

A {\em graph} is a pair  $G=(V,E)$, where $V$ is
a finite set of {\em vertices}
and $E \subseteq \{ \{u,v\} \mid u,v \in
V,~u \not= v\}$ is a finite set of {\em edges}.
For a vertex $v\in V$, the  set $N(v)=\{u\in V \mid \{v,u\}\in E\}$ 
is called the {\em set of all neighbors} of $v$ 
or {\em neighborhood} of $v$.

We will use the following indexed graphs.
\begin{itemize}
\item
By 
$P_n=(\{v_1,\ldots,v_n\},\{ \{v_1,v_2\},\ldots, \{v_{n-1},v_n\}\})$, $n \ge 2$,
we denote the  path on $n$ vertices.

\item
By  
$C_n=(\{v_1,\ldots,v_n\},\{\{v_1,v_2\},\ldots, \{v_{n-1},v_n\},\{v_1,v_n\}\})$, $n \ge 3$,
we denote the  cycle on $n$ vertices.

\item 
By $K_n=(\{v_1,\ldots,v_n\},\{ \{v_i,v_j\}\mid 1\leq i<j\leq n\})$, $n \ge 1$,  we denote the complete graph on $n$ vertices.

\item
By $K_{n,m}=(\{v_1,\ldots,v_n,w_1,\ldots,w_m\},\{ \{v_i,w_j\}\mid 1\leq i\leq n,1\leq  j\leq m\})$,
$n,m \ge 1$ we denote the complete bipartite graph with $n+m$ vertices.
\end{itemize}

\subsection{Digraphs}

A {\em directed graph} or {\em digraph} is a pair  $G=(V,E)$, where $V$ is
a finite set of {\em vertices} and
$E\subseteq \{(u,v) \mid u,v \in V,~u \not= v\}$ is a finite set of ordered pairs of distinct
vertices called {\em arcs} or {\em directed edges}.
For a vertex $v\in V$, the sets $N^+(v)=\{u\in V \mid (v,u)\in E\}$ and
$N^-(v)=\{u\in V \mid (u,v)\in E\}$ are called the {\em set of all successors}
and the {\em set of all  predecessors} of $v$. The set $N(v)=N^+(v) \cup N^-(v)$
is  the {\em set of all neighbors}.
The  {\em outdegree} of $v$, $\odeg(v)$ for short, is the number
of successors of $v$ and the  {\em indegree} of $v$, $\ideg(v)$ for short,
is the number of predecessors of $v$.
The {\em maximum (vertex) degree} is defined by $\Delta(G)=\max_{v\in V} (\odeg(v)+\ideg(v))$.

For some given digraph $G=(V,E)$, we define
its underlying undirected graph by ignoring the directions of the arcs, i.e.
$\un(G)=(V,\{\{u,v\} \mid (u,v)\in E, u,v\in V\})$.
For some (di)graph class $F$ we define
$\free(F)$ as the set of all (di)graphs $G$ such that no induced sub(di)graph
of $G$ is isomorphic to a member of  $F$.

A digraph $G'=(V',E')$ is a {\em subdigraph} of digraph $G=(V,E)$ if $V'\subseteq V$
and $E'\subseteq E$.  If every arc of $E$ with both end vertices in $V'$  is in
$E'$, we say that $G'$ is an {\em induced subdigraph} of $G$ and we
write $G'=G[V']$.

An {\em oriented graph} is a digraph with no loops and no opposite arcs.
We will use the following indexed oriented graphs.
\begin{itemize}
\item
By 
$\overrightarrow{P_n}=(\{v_1,\ldots,v_n\},\{ (v_1,v_2),\ldots, (v_{n-1},v_n)\})$, $n \ge 2$,
we denote the oriented path on $n$ vertices.
\item
By  
$\overrightarrow{C_n}=(\{v_1,\ldots,v_n\},\{(v_1,v_2),\ldots, (v_{n-1},v_n),(v_n,v_1)\})$, $n \ge 2$,
we denote the oriented cycle on $n$ vertices.

\item
By 
$\overrightarrow{K_{n,m}}=(\{v_1,\ldots,v_n,w_1,\ldots,w_m\},\{ (v_i,w_j)\mid 1\leq i\leq n,1\leq  j\leq m\})$,
$n,m \ge 1$ we denote an oriented complete bipartite digraph with $n+m$ vertices.
\end{itemize}
An {\em oriented forest (tree)} is an orientation of a forest (tree).
An {\em out-rooted-tree} ({\em in-rooted-tree}) is an orientation of a
tree with a distinguished root such that
all arcs are directed away from (directed to) the root.
A {\em directed acyclic graph (DAG for short)} is a digraph without any oriented cycle $\overrightarrow{C_n}$,
for $n\geq 2$, as subdigraph. A {\em tournament} is a digraph in which there is exactly one edge 
between every two distinct vertices.

A vertex $v$ is {\em reachable} from vertex $u$ in $G$, if $G$ contains
an oriented path $\overrightarrow{P_n}$ as a subdigraph having start vertex $u$ and
end vertex $v$.
A {\em topological ordering} of a directed graph is a linear ordering of its vertices
such that for every directed edge $(u,v)$, vertex $u$ is before vertex $v$ in the ordering.
A  digraph $G$ is bipartite if $\un(G)$ is bipartite and
a digraph $G$ is planar if $\un(G)$ is planar.

A digraph $G=(V,E)$ is {\em transitive} if for
every pair $(u,v)\in E$ and $(v,w)\in E$ of arcs
with $u\neq w$ the arc $(u,w)$ also belongs to $E$.
The {\em transitive closure} $tc(G)$ of a digraph $G$ has the same
vertex set as $G$ and for two distinct vertices $u,v$ there is
an arc $(u,v)$ in $tc(G)$ if and only if
$v$ is reachable from  $u$ in $G$.

\subsection{Coloring undirected graphs}\label{co-und}

\begin{definition}[Graph coloring]\label{def-oricol}
An \emph{$r$-coloring} of a graph $G=(V,E)$  is a mapping $c:V\to \{1,\ldots,r\}$
such that:
\begin{itemize}
	\item $c(u)\neq c(v)$ for every $\{u,v\}\in E$.\footnote{The single condition on the mapping 
	to be a feasible coloring will be extended for oriented colorings in Definition \ref{def-oc}.}
\end{itemize}
The  {\em chromatic number} of $G$, denoted by $\chi(G)$, is the smallest integer $r$
such that $G$ has a $r$-coloring.
\end{definition}

We consider the following decision problem.

\begin{desctight}
\item[Name]     Chromatic Number (CN)

\item[Instance] A graph $G=(V,E)$ and a positive integer $r \leq |V|$.

\item[Question] Is there a $r$-coloring for $G$?
\end{desctight}

If $r$ is a constant, i.e.\ not part of the input, the corresponding problem
is denoted by $r$-Chromatic Number ($\CN_{r}$).
Even on 4-regular planar graphs $\CN_{3}$ is NP-complete  \cite{Dai80}.

It is well known that bipartite graphs are exactly the
graphs which allow a 2-coloring and that planar graphs are graphs that allow a 4-coloring.
On undirected co-graphs, the Chromatic Number problem is easy to solve
by the following result proven by Corneil et al.:

\begin{lemma}[\cite{CLS81}]\label{colo-und}
Let $G_1$ and $G_2$ be two vertex-disjoint graphs.
\begin{enumerate}
\item $\chi_o((\{v\},\emptyset))=1$
\item $\chi(G_1\cup G_2) = \max(\chi(G_1),\chi(G_2))$

\item $\chi(G_1\times G_2)= \chi(G_1) +\chi(G_2)$
\end{enumerate}
\end{lemma}

\begin{proposition} Let $G$ be a co-graph. Then,
$\chi(G)$  can be computed in linear time.
\end{proposition}

Coloring a graph $G=(V,E)$ can be done by a greedy algorithm.
For some given ordering $\pi$ of $V$, the vertices are ordered as a sequence in which each vertex is assigned to the minimum possible value that is not forbidden by the colors of its neighbors, see Algorithm \ref{algo-bip}.
Obviously, different orders can lead to
different numbers of colors. But there is always an ordering yielding to the minimum
number of colors,  which is hard to find in general.

\begin{algorithm}[ht]
\KwData{A graph $G=(\{v_1,\ldots,v_n\},E)$ and an ordering $\pi: v_1 < \hdots < v_n$ of its vertices.}
\KwResult{An admitted vertex coloring $c:\{v_1,\ldots,v_n\} \mapsto \NN $ of $G$.}
\For{($i=1$ to $n$)}{
$c(v_i)=\infty$
}
$c(v_1)=1$;\\
\For{($i=2$ to $n$)}{
$c(v_i)=\min\{\NN \setminus  \{c(v) \mid v\in N(v_i)\}\}$ \tcc*[r]{$\NN$ denotes the set of all positive integers.}
}
\caption{{\sc Greedy Coloring}}
\label{algo-bip}
\end{algorithm}

%


The class of perfectly orderable graphs consists of those graphs for which the 
given greedy algorithm leads to an ordering  yielding to an optimal coloring, 
not only for the graph itself but also for all of its induced subgraphs.

\begin{definition}[Perfectly orderable graph \cite{Chv84}]\label{def-peog}
Let $G=(V,E)$ be a graph. A linear ordering on $V$ is {\em perfect} if a greedy
coloring algorithm with that ordering optimally colors every induced subgraph of $G$.
A graph $G$ is {\em perfectly orderable} if it admits a perfect order.
\end{definition}

\begin{theorem}[\cite{Chv84}]\label{th-pe}
A linear ordering $\pi$ of a graph $G$ is perfect if and only if
there is no induced $$P_4=(\{a,b,c,d\},\{\{a,b\},\{b,c\},\{c,d\}\})$$
in $G$ such that $\pi(a)<\pi(b)$, $\pi(b)<\pi(c)$, and $\pi(d)<\pi(c)$.
\end{theorem}


\begin{example}\label{ex-pe}
Every co-graph is perfectly orderable, since it does not have any
induced $P_4$.
\end{example}

The Chromatic Number problem can be solved by
an $\fpt$-algorithm w.r.t.\ the  tree-width of the 
input graph \cite{Gur08c}.
In contrast, this is not true for clique-width, since
it has been shown  in \cite{FGLS10a}, that the Chromatic Number problem
is $\w[1]$-hard w.r.t.\ the  clique-width of the input graph.
That is, under reasonable assumptions an $\xp$-algorithm is the best one can hope for.
Such algorithms are known, see \cite{EGW01a,KR01}.

In order to show  fixed parameter tractability for  $r$-Chromatic Number w.r.t.\ the
parameter clique-width
one can use its defineability within monadic second order logic \cite{CMR00}.

\subsection{Coloring oriented graphs}

Oriented graph coloring has been introduced by Courcelle \cite{Cou94} in 1994.
We consider oriented graph coloring on oriented graphs,
i.e. digraphs with no loops and no opposite arcs.

\begin{definition}[Oriented graph coloring]\label{def-oc}
Let $G=(V,E)$ be an oriented graph. 
An \emph{oriented $r$-coloring} of $G$ is a mapping $c:V\to \{1,\ldots,r\}$
such that:
\begin{itemize}
	\item $c(u)\neq c(v)$ for every $(u,v)\in E$,
	\item $c(u)\neq c(y)$ for every two arcs $(u,v)\in E$ and $(x,y)\in E$ with $c(v)=c(x)$.
\end{itemize}
The {\em oriented chromatic number} of $G$, denoted by $\chi_o(G)$, is the smallest integer $r$
such that $G$ has an oriented $r$-coloring.
The vertex sets $V_i=\{v\in V\mid c(v)=i\}$, with $1\leq i\leq r$, divide
$V$ into a partition of so called {\em color classes}.
\end{definition}
f
For two oriented graphs $G_1=(V_1,E_1)$ and $G_2=(V_2,E_2)$ a {\em homomorphism}
from $G_1$ to $G_2$, $G_1 \to G_2$ for short,
is a mapping $h: V_1 \to V_2$ such that  $(u,v) \in E_1$ implies
$(h(u),h(v)) \in E_2$.
A homomorphism from $G_1$ to $G_2$ can be regarded as an oriented coloring of $G_1$
that uses the vertices of $G_2$ as colors classes. 
Therefore, digraph $G_2$ is called the {\em color graph} of $G_1$.
This leads to equivalent definitions for the
oriented coloring and the oriented chromatic number.
There is an oriented $r$-coloring of an oriented graph $G_1$
if and only if there is a homomorphism from $G_1$ to some oriented graph $G_2$ with $r$ vertices.
Thus, the oriented chromatic number of $G_1$ is the minimum number of vertices in an
oriented graph $G_2$ such that there is a homomorphism from $G_1$ to $G_2$.
Obviously, it is possible to choose $G_2$ as a tournament.

\begin{observation}\label{oc-tournament}
There is an oriented $r$-coloring of an oriented graph $G_1$
if and only if there is a homomorphism from $G_1$ to some tournament $G_2$ with $r$ vertices.
Further, the oriented chromatic number of $G_1$ is the minimum number of vertices in a
tournament $G_2$ such that there is a homomorphism from $G_1$ to $G_2$.
\end{observation}

\begin{observation}\label{obs-low}
Let $G$ be an  oriented graph. Then, it holds 
that
$\chi(\un(G))\leq \chi_o(G)$.
\end{observation}

On the other hand it is not possible to bound the oriented chromatic number 
of an oriented graph $G$ by a function of the (undirected)
chromatic number of $\un(G)$. This has been shown in \cite[Section 3]{Sop16} 
by an orientation $K'_{n,n}$ of a $K_{n,n}$ satisfying $\chi_o(K'_{n,n})=2n$ and 
$\chi(\un(K'_{n,n}))=2$.

\begin{lemma}\label{le-col-subdigraph}
Let $G$ be an oriented graph and $H$ be a subdigraph
of $G$. Then, an oriented $r$-coloring of $G$ leads also an oriented
$r$-coloring for $H$.
\end{lemma}

\begin{corollary}\label{le-isubdigraph}
Let $G$ be an oriented graph and $H$ be a subdigraph
of $G$. Then, it holds that $\chi_o(H)\leq \chi_o(G)$.
\end{corollary}

\begin{example}\label{ex-color}
For oriented paths and oriented cycles we know:
$\chi_o(\overrightarrow{P_2})=2$,  $\chi_o(\overrightarrow{P_3})=3$,
$\chi_o(\overrightarrow{C_4})=4$, $\chi_o(\overrightarrow{C_5})=5$.
\end{example}

An oriented graph $G=(V,E)$ is an {\em oriented clique} ({\em o-clique})
if $\chi_o(G) = |V|$.
Thus all graphs given in Example~\ref{ex-color} are oriented cliques.
Further by Observation \ref{obs-low} every tournament is an  oriented clique.
Thus, for DAGs the oriented chromatic number is unbounded.

We consider the following decision problem.

\begin{desctight}
\item[Name]      Oriented Chromatic Number ($\OCN$)
\item[Instance]  An oriented graph $G=(V,E)$ and a positive integer $r \leq |V|$.
\item[Question]  Is there an oriented $r$-coloring for $G$?
\end{desctight}

If $r$ is constant, i.e.\  not part of the input, the corresponding problem
is denoted by $r$-Oriented Chromatic Number ($\OCN_{r}$).
If  $r\leq 3$, then  $\OCN_{k}$ can be decided in polynomial time, while $\OCN_{4}$ is NP-complete \cite{KMG04}. 
$\OCN_{4}$ is even known to be NP-complete for several restricted
classes of digraphs, e.g.  for bounded degree DAGs \cite{CD06},  bounded degree bipartite
oriented graph \cite{CD06}, graphs with K-width 1 and DAG-depth 3 \cite{GH10}, 
DAGs of K-width 3 and DAG-depth 5 \cite{GHKLOR14}, digraphs of DAG-width 2, K-width 1 and DAG-depth 3 
\cite{GHKLOR14}, and acyclic oriented graphs whose underlying graph is connected, planar, bipartite 
and has maximum degree 3   \cite{CFGK16}.

Up to now, the definition of oriented coloring was frequently
applied to undirected graphs. For an undirected graph $G$ the
maximum value $\chi_o(G')$ of all possible orientations $G'$ of
$G$ is considered. In this sense, every tree  has oriented chromatic number at most $3$ and
every cycle $C_n$ has oriented chromatic number at most $5$.
For several further graph classes there exist bounds on the oriented chromatic number.
Among these are outerplanar graphs \cite{Sop97}, Halin graphs \cite{DS14}, and planar graphs \cite{Mar13}.
See \cite{Sop16} for a survey.


To advance research in this field, we consider oriented graph
coloring on recursively defined oriented graph classes.

\section{Coloring transitive acyclic digraphs}\label{sec-ta}

In this section we will use the concept of
perfectly orderable graphs and Theorem \ref{th-pe}
in order to find oriented colorings of transitive acyclic digraphs.

\begin{theorem}\label{algop2}
Let $G$ be a transitive acyclic digraph. Then,  every greedy coloring along a topological ordering of $G$
leads to an optimal oriented coloring of $G$ and
$\chi_o(G)$  can be computed in linear time.
\end{theorem}

\begin{proof}
Let $G$ be a transitive acyclic digraph.
Since $G=(V,E)$ is  acyclic  there is a topological ordering $t$ for $G$.
Since $G$ is transitive, it does not contain the  following orientation
of a $P_4$ as  an induced subdigraph.
$$\bullet \rightarrow \bullet \rightarrow\bullet\leftarrow \bullet$$
By Theorem \ref{th-pe}  every linear ordering and thus also $t$ is perfect on the vertex set of 
graph $\un(G)=(V,E_u)$.
Let $c:V\to \{1,\ldots,k\}$ be a coloring for $\un(G)$ obtained by the greedy
algorithm (Algorithm \ref{algo-bip}) for $t$ on $V$. We
show that $c$ is an oriented coloring for $G$ by verifying the two properties of 
Definition \ref{def-oc}.
\begin{itemize}
	\item Property $c(u)\neq c(v)$ holds for every $(u,v)\in E$ since  $c(u)\neq c(v)$ holds for every $\{u,v\}\in E_u$.

	\item Property $c(u)\neq c(y)$ for every two arcs $(u,v)\in E$ and $(x,y)\in E$ with $c(v)=c(x)$ 
	holds by the following argumentation.

	Assume there is an arc $(v_i,v_j)\in E$ with $v_i<v_j$ in $t$ but
$c(v_i)>c(v_j)$. Then, when coloring $v_i$ we would have taken $c(v_j)$
if possible, as we always take the minimum possible color value.
Since this was not possible there must have been an other vertex $v_k<v_i$
which was colored before $v_i$ with $c(v_k)=c(v_j)$ and $(v_k,v_i)\in
E$. But if $(v_k,v_i)\in E$ and $(v_i,v_j)\in E$, due to transitivity it
must also hold that $(v_k,v_j)\in E$ and consequently, $c(v_k)=c(v_j)$ is not
possible. Thus, the assumption was wrong and for every arc $(v_i,v_j)\in
E$ with $v_i<v_j$ in $t$ it must hold that $c(v_i)<(c_j)$.

\end{itemize}
The optimality of oriented coloring $c$ follows since the lower bound of Observation \ref{obs-low}
is achieved.
\end{proof}

In order to state the next result, let $\omega(G)$
be the number of vertices in a largest clique in the (undirected) graph $G$.

\begin{corollary}\label{cor-tt-a}
Let $G$ be a transitive
acyclic digraph. Then, it holds that
$$\chi_o(G)=\chi(\un(G))=\omega(\un(G))$$ and all three values can be
computed in linear time.
\end{corollary}

For some oriented graph $G$ we denote by $\ell(G)$   the length of a longest oriented
path in $G$.

\begin{proposition}\label{pro1}
Let $G$ be a transitive acyclic digraph. Then, it holds that $\chi_o(G)= \ell(G)+1$.
\end{proposition}

\begin{proof}
The proof of Theorem \ref{algop2} leads to an optimal oriented coloring using $\ell(G)+1$ colors.
\end{proof}

Next, we consider oriented colorings of oriented graphs with bounded vertex degree.
For every  oriented graph the oriented chromatic number
can be bounded (exponentially) by its  maximum vertex degree $\Delta$  \cite{KSZ97}.
For small vertex degrees $\Delta\leq 7$ there are better bounds in \cite{Duf19,DOPS20}.

\begin{corollary}\label{cor-tt-av}
Let $G$ be a transitive acyclic digraph. Then, it holds that $\chi_o(G)\leq \Delta(G)+1$.
\end{corollary}

\begin{proof}Let $G$ be a transitive acyclic digraph.
By Proposition \ref{pro1} and the fact that the first vertex of a longest path within
a transitive digraph $G$ has outdegree at least $\ell(G)$, it follows that the oriented chromatic number of $G$ can be estimated by $\chi_o(G)= \ell(G)+1 \leq \Delta(G)+1$.
\end{proof}

\begin{proposition}\label{pro2}
Let $G$ be an acyclic digraph. Then, it holds that $\chi_o(G)\leq \ell(G)+1$.
\end{proposition}

\begin{proof}
Let $G$ be an  acyclic digraph and $G'$ its transitive closure.
Using Corollary \ref{le-isubdigraph} and Proposition \ref{pro1} we know that
$\chi_o(G)\leq \chi_o(G')= \ell(G')+1=\ell(G)+1$.
\end{proof}

Now, we consider oriented graph
coloring on recursively defined oriented graph classes.

\section{Coloring oriented  co-graphs}\label{sec-cog}

We recall operations 
which have been considered  by Bechet et al.\ in \cite{BGR97}.
Let $G_1=(V_1,E_1)$ and  $G_2=(V_2,E_2)$ be two vertex-disjoint digraphs.
\begin{itemize}
\item
The {\em disjoint union} of $G_1$ and $G_2$,
denoted by $G_1 \oplus G_2$,
is the digraph with vertex set $V_1\cup V_2$ and
arc set $E_1\cup E_2$.

\item
The {\em order composition} of $G_1$ and  $G_2$,
denoted by $G_1\oslash  G_2$,
is defined by their disjoint union plus all possible arcs from
vertices of $G_1$ to vertices of $G_2$.
\end{itemize}

By  omitting the series composition
within the definition  of directed co-graphs in \cite{CP06}, we obtain the class of all
{\em oriented  co-graphs}.

\begin{definition}[Oriented co-graphs]\label{dcog}
The class of {\em oriented complement reducible graphs}, {\em oriented co-graphs} for short,
is recursively defined as follows.
\begin{enumerate}
\item Every digraph on a single vertex $(\{v\},\emptyset)$,
denoted by $v$, is an {\em oriented co-graph}.

\item If  $G_1$ and $G_2$  are two vertex-disjoint oriented co-graphs, then
\begin{enumerate}
\item
the disjoint union
$G_1\oplus G_2$, and
\item
the order composition
$G_1\oslash G_2$  are {\em oriented co-graphs}.
\end{enumerate}
\end{enumerate}
The class of oriented co-graphs is denoted by $\OC$.
\end{definition}

Every expression $X$ using  the operations of Definition  \ref{dcog}
is called a {\em di-co-expression}. Example \ref{ex-orico} illustrates these notations.

\begin{example}\label{ex-orico}
The following di-co-expression $X$ defines the oriented graph  shown in Figure \ref{F03}.
$$
X=((v_1\oslash v_3) \oslash(v_2 \oplus  v_4)) \label{eq-ori-c4x}
$$
\end{example}

\begin{figure}[hbtp]
\centering
\centerline{\includegraphics[width=0.25\textwidth]{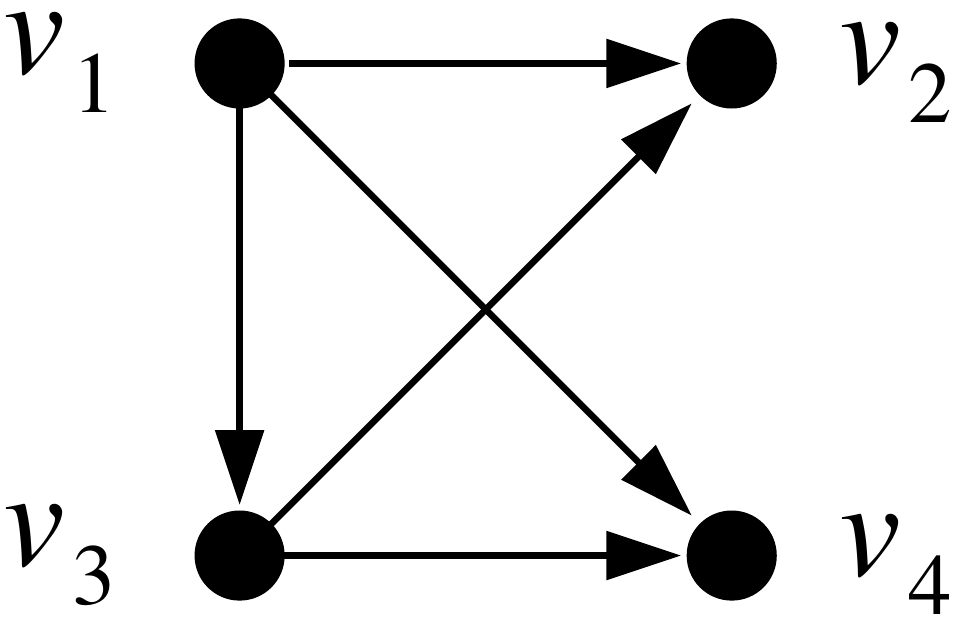}}
\caption{Digraph defined by di-co-expression $X$ in Example \ref{ex-orico}.}
\label{F03}
\end{figure}

Several classes of digraphs are included in the set of all oriented co-graphs.

\begin{proposition} 
Every transitive tournament is an oriented co-graph.
\end{proposition}

\begin{proposition} 
Every oriented bipartite graph $\overrightarrow{K_{n,m}}$ is an oriented co-graph.
\end{proposition}

The set of all oriented co-graphs is  closed under taking induced subdigraphs.
Using the notations of \cite{VTL82} we denote the  following orientation
of a $P_4$ as the $N$ graph.
$$N=\bullet \rightarrow \bullet \leftarrow\bullet\rightarrow \bullet$$
The class of oriented co-graphs  can
be characterized by
excluding the four forbidden induced subdigraphs
$\overleftrightarrow{P_2}=(\{u,v\},\{(u,v),(v,u)\})$, $\overrightarrow{P_3}$, $\overrightarrow{C_3}$, and $N$, see \cite{GKR21b}.
The class of oriented co-graphs has already been analyzed by Lawler
in \cite{Law76} and Corneil et al.\ in \cite{CLS81} using the notation
of {\em transitive series-parallel (TSP) digraphs}.

For every oriented co-graph we can define a tree structure,
denoted as {\em di-co-tree}.
The leaves of the di-co-tree represent the
vertices of the digraph and the inner vertices of the di-co-tree  correspond
to the operations applied on the subexpressions defined by the subtrees.
For every oriented co-graph one can construct a di-co-tree in linear time,
see \cite{CP06}.
Oriented co-graphs are a subclass of directed co-graphs \cite{CP06} and both classes 
are interesting from an algorithmic point of view
since several hard graph problems can be solved in
polynomial time  by dynamic programming along the tree structure of
the input graph, see \cite{Ret98,BM14,Gur17a,GR18c,GKR19f,GKR19d,GHKRRW20,GKRRW20}.

\begin{lemma}[\cite{GKR19d}]\label{le-co-gr}Let $G_1$ and $G_2$ be two vertex-disjoint oriented co-graphs.
Then, the following equations hold.
\begin{enumerate}
\item $\chi_o((\{v\},\emptyset)) = 1$
\item $\chi_o(G_1 \oplus G_2)    = \max(\chi_o (G_1),  \chi_o(G_2))$
\item $\chi_o(G_1\oslash G_2)    = \chi_o(G_1) +  \chi_o(G_2)$
\end{enumerate}
\end{lemma}

\begin{theorem}[\cite{GKR19d}]\label{algop}
Let $G$ be an oriented co-graph. Then, an optimal oriented coloring for $G$ and
$\chi_o(G)$  can be computed in linear time.
\end{theorem}

The result in \cite{GKR19d} concerning the oriented coloring on oriented co-graphs is
based on a dynamic
programming along a di-co-tree for the given oriented co-graph as input.
Since every oriented co-graph is transitive and acyclic,
Theorem \ref{algop2} leads to the next result, which re-proves
Theorem \ref{algop}.

\begin{corollary}\label{c-co}
Let $G$ be an oriented co-graph. Then,  every greedy coloring along a topological ordering of $G$
leads to an optimal oriented coloring of $G$ and
$\chi_o(G)$  can be computed in linear time.
\end{corollary}

Theorem \ref{algop2} is more general than
Corollary \ref{c-co} since it does not exclude $N$
which is a
forbidden induced subdigraph for oriented co-graphs. It holds that
$$\OC=\free\{\overleftrightarrow{P_2},\overrightarrow{P_3}, \overrightarrow{C_3}, N\}\subseteq \free\{\overleftrightarrow{P_2},\overrightarrow{P_3}, \overrightarrow{C_3}\}$$
and $\free\{\overleftrightarrow{P_2},\overrightarrow{P_3}, \overrightarrow{C_3}\}$ is equivalent to the set of all acyclic transitive digraphs.

Since every oriented co-graph is transitive and acyclic, Corollary
\ref{cor-tt-av} leads to the following bound.

\begin{corollary}\label{th-degxxco2}
Let $G$ be an oriented co-digraph. Then, it holds that $\chi_o(G)\leq \Delta(G)+1$.
\end{corollary}

There are classes of oriented co-graphs, e.g., the class of all
$\overrightarrow{K_{1,n}}$, for which the  oriented
chromatic number is even bounded by a constant and thus smaller than the
shown bound. The set of all
transitive tournaments shows that the
bound given in Corollary \ref{th-degxxco2} is best possible.

\section{Coloring msp-digraphs}

We recall the definitions from \cite{BG18} which are
based on \cite{VTL82}. First, we introduce two operations for two 
vertex-disjoint digraphs $G_1=(V_1,E_1)$ and $G_2=(V_2,E_2)$.
Let $O_1$ be the set of vertices of outdegree $0$ (set of sinks) in $G_1$ and
$I_2$ be the set of vertices of indegree $0$ (set of sources) in $G_2$.

\begin{itemize}
\item
The {\em parallel composition} of $G_1$ and $G_2$,
denoted by $G_1 \cup G_2$, is the digraph with vertex set $V_1\cup V_2$
and arc set $E_1\cup E_2$.

\item
The {\em  series composition}  of $G_1$ and $G_2$,
denoted by $G_1 \times G_2$ is the digraph with vertex set $V_1\cup V_2$
and arc set $E_1\cup E_2\cup  \{(v,w)\mid v\in O_1, w\in I_2\}$.
\end{itemize}

\begin{definition}[Msp-digraphs]\label{def-msp}
The class of {\em minimal series-parallel digraphs}, {\em msp-di\-graphs} for
short, is recursively defined as follows.
\begin{enumerate}
\item Every digraph on a single vertex $(\{v\},\emptyset)$,
denoted by $v$, is a {\em minimal series-parallel digraph}.

\item If  $G_1$ and $G_2$
are vertex-disjoint minimal series-parallel digraphs then,
\begin{enumerate}
\item the parallel composition
$G_1 \cup G_2$ and

\item  then series composition
$G_1 \times G_2$
are  {\em minimal series-parallel digraphs}.
\end{enumerate}
\end{enumerate}
The class of minimal series-parallel digraphs is denoted as $\MSP$.
\end{definition}

Minimal (vertex) series-parallel digraphs are the line digraphs of 
edge series-parallel digraphs  \cite{VTL82}, which are an oriented version of
the well known class of series-parallel graphs.

Every expression $X$ using  the operations of Definition \ref{def-msp}
is called an {\em msp-expression}.
The digraph defined by the expression $X$ is denoted by $\g(X)$.
We  illustrate this by two expressions which we will refer to later.

\begin{example}\label{ex-msp}
The following msp-expressions $X_1$ and $X_2$ define msp-digraphs on five and six vertices
shown in Figure~\ref{F06} and Figure~\ref{F03y}.
$$
X_1=  (v_1\times((v_2\times v_3)\cup v_4))\times v_5
$$
$$
X_2=  (v_1\times(((v_2\times v_3)\times v_4)\cup v_5))\times v_6
$$
\end{example}

\begin{figure}[hbtp]
\centering
\parbox[b]{67mm}{
\centerline{\includegraphics[width=0.4\textwidth]{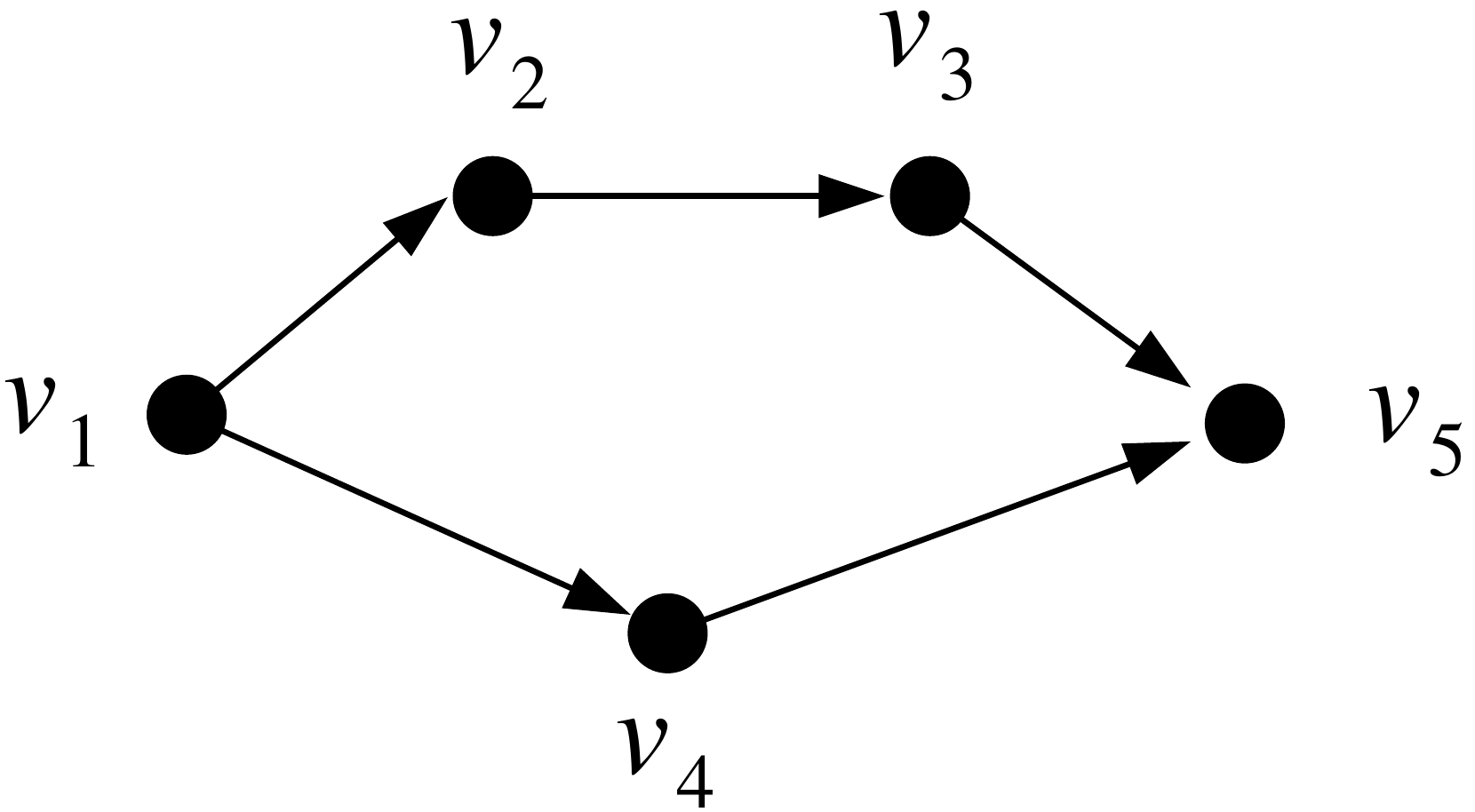}}
\caption{$\gr(X_1)$ in Example \ref{ex-msp}.
}
\label{F06}}
~~~~~~~~~
\parbox[b]{67mm}{
\centerline{\includegraphics[width=0.5\textwidth]{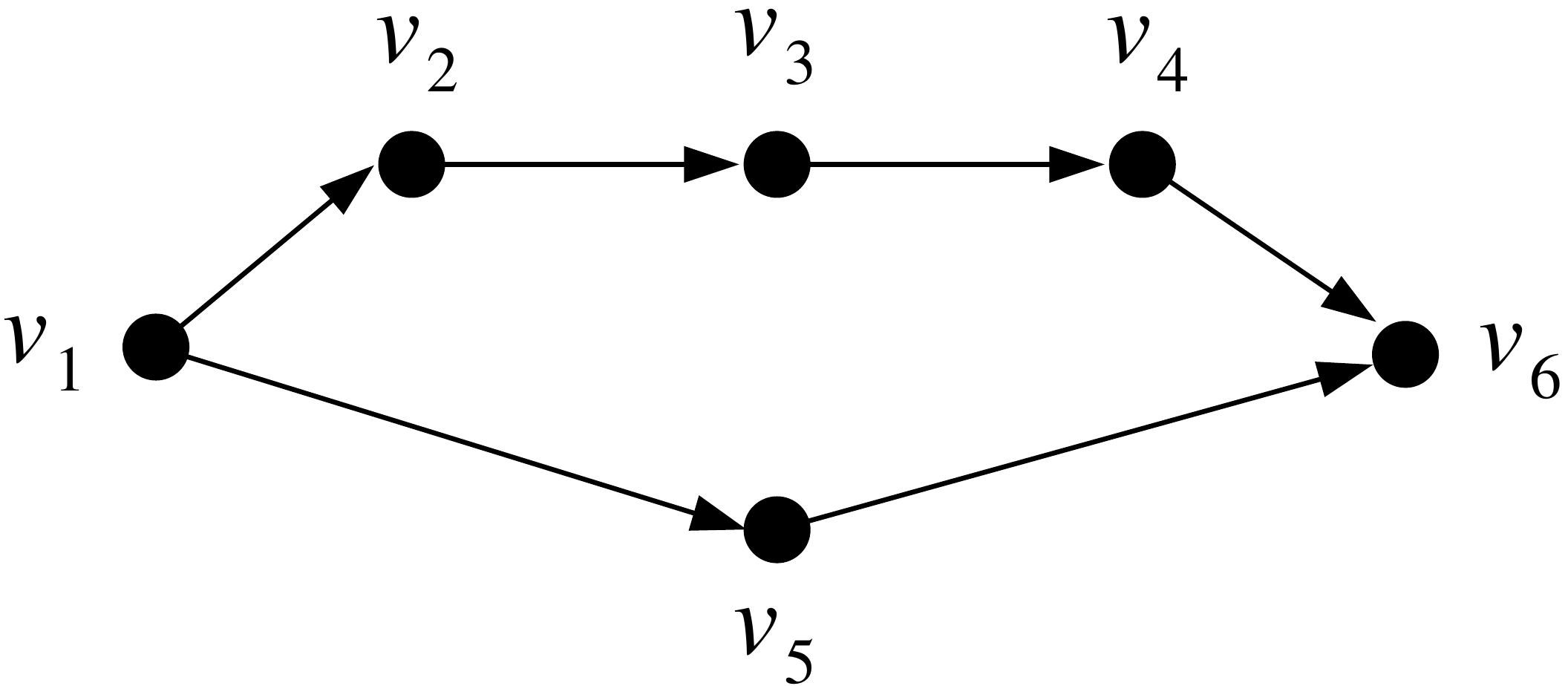}}
\caption{$\gr(X_2)$ in Example \ref{ex-msp}.}
\label{F03y}}
\end{figure}

By removing vertex $v_3$ from  $\g(X_2)$ in Example \ref{ex-msp},
we obtain an oriented graph which is no msp-digraph. This implies
that the set of all msp-digraphs is not closed under taking induced subdigraphs.

A further remarkable property of msp-digraphs is that
the oriented chromatic number of the disjoint union of two msp-digraphs can
be larger than the maximum oriented chromatic number of the involved digraphs.
This follows by the digraphs defined by expressions $X_1$ and $X_2$ in Example \ref{ex-six2}, which both
have oriented chromatic number $4$ but their disjoint union leads to a digraph with
oriented chromatic number $5$.

\begin{example}\label{ex-six2} In the following two msp-expressions we assume that
the  series composition $\times$ binds more strongly than the parallel composition $\cup$.
$$X_1=v_{1} \times (v_2 \cup v_3 \times v_4) \times v_5 \times v_6$$
$$X_2=w_{1} \times (w_2 \cup w_3 \times (w_{4} \cup w_{5} \times w_{6} ))\times w_{7}$$
\end{example}

Several classes of digraphs are included in the set of all msp-digraphs.

\begin{proposition} 
Every  in- and out-rooted tree is an msp-digraph.
\end{proposition}

\begin{proposition} 
Every oriented bipartite graph $\overrightarrow{K_{n,m}}$ is an msp-digraph.
\end{proposition}

For every  msp-digraph we can define a tree structure $T$, which is
denoted as {\em msp-tree}. (In \cite{VTL82}, the tree-structure for an msp-digraphs
is denoted as binary decomposition tree.)
The leaves of an msp-tree represent the
vertices of the digraph and the inner vertices of the msp-tree  correspond
to the operations applied on the subexpressions defined by the subtrees.
For every  msp-digraph one can construct an msp-tree in linear time,
see \cite{VTL82}.

Next we want to give an algorithm to compute the oriented chromatic
number of an msp-digraph. Considering the solutions of Sections \ref{sec-ta} and \ref{sec-cog}
we conclude that 
a greedy coloring of $\un(G)$ along a topological ordering of $G$ does not
work for computing the oriented chromatic number of an msp-digraph $G$. 
An oriented path would be
colored by only two colors which is not an admitted oriented coloring.
Further, a dynamic programming solution using similar formulas to 
Lemma \ref{le-co-gr} is not possible for computing the oriented chromatic number of msp-digraphs.
Example \ref{ex-six2} implies that the oriented chromatic number of the disjoint union of two msp-digraphs can
be larger than the maximum oriented chromatic number of the involved digraphs.

In order to give an algorithm to compute the oriented chromatic
number of msp-digraphs, we first show that this value can be bounded by
a constant.

The  class of undirected series-parallel  graphs was considered in  \cite{Sop97}
by showing that every orientation of a series-parallel  graph has oriented
chromatic number at most 7. This bound can not be applied to  msp-digraphs,
since  the set of all $\overrightarrow{K_{n,m}}$  is a subset of msp-digraphs
and the underlying graphs are even of unbounded tree-width \cite{Bod98} and thus, no series-parallel  graphs.

Nevertheless, we can show that $7$ is also an upper bound for the oriented chromatic
number of msp-digraphs. Therefore we introduce recursively defined oriented
graphs $M_i$.

\begin{lemma}\label{le-mi}
We recursively define oriented digraphs $M_i$ as follows. 
$M_0$ is a single vertex graph and for $i\geq 1$
we define
$$M_i=M_{i-1}\cup M_{i-1}\cup (M_{i-1} \times M_{i-1}).$$
Then, every msp-digraph $G$ is a (-n induced) subdigraph of some $M_i$ such that every source in $G$
is a source in $M_i$ and every sink in $G$
is a sink in $M_i$.
\end{lemma}

\begin{proof}
The lemma can be shown by induction over the number of vertices in some msp-digraph $G$.
If $G$ has exactly one vertex, the claim holds true by choosing $M_0$.

Next, assume that $G$ has $k>1$ vertices. Then, it holds that $G=G_1 \Box G_2$ for some
$\Box\in\{\cup,\times\}$ and $G_1$ and $G_2$ are msp-digraphs with less than $k$ vertices.
By the induction hypothesis we conclude that there are two integers $i_1$ and $i_2$ such that
digraph $G_1 \Box G_2$ is a subdigraph of digraph $M_{i_1} \Box M_{i_2}$. (If $\Box=\times$, it is important
that every sink in $G_1$
is a  sink in $M_{i_1}$ and every source in $G_2$
is a source in $M_{i_2}$.)
Thus, $G$ is a subdigraph of digraph $M_{i_1}\cup M_{i_2}\cup (M_{i_1} \times M_{i_2})$. W.l.o.g. we
assume that $i_1\leq i_2$. By construction it follows that $M_{i_1}$ is a
subdigraph of $M_{i_2}$. Consequently, $G$ is a
subdigraph of digraph $M_{i_2}\cup M_{i_2}\cup (M_{i_2} \times M_{i_2})=M_{i_2+1}$ and every source in $G$
is a source in $M_{i_2+1}$ and every sink in $G$
is a sink in $M_{i_2+1}$.
This completes the proof of the claim.
\end{proof}

\begin{theorem}\label{algspd}
Let $G$ be an msp-digraph. Then, it holds that $\chi_o(G)\leq 7$.
\end{theorem}

\begin{proof}
By Lemma \ref{le-col-subdigraph} we can show the theorem by coloring the 
digraphs $M_i$ defined in Lemma \ref{le-mi}.
Further, the first two  occurrences of $M_{i-1}$ in $M_i$ can be colored
in the same way. Thus, we can restrict to
oriented graphs $M'_i$ which are defined
as follows. $M'_0$ is a single vertex graph and for $i\geq 1$
we define
$$M'_i=M'_{i-1}\cup (M'_{i-1} \times M'_{i-1}).$$

We define an oriented $7$-coloring $c$ for $M'_i$ as follows.
For some  vertex $v$ of  $M'_i$ we define by $c(v,i)$ the color of $v$ in  $M'_i$.
First, we color $M'_0$  by assigning color $0$ to the single
vertex in $M'_0$.\footnote{Please note that using colors starting at value $0$
instead of $1$ does not contradict Definition \ref{def-oc}.} For $i\geq 1$
we define the colors for the vertices $v$ in $M'_i=M'_{i-1}\cup (M'_{i-1} \times M'_{i-1})$
according to the three copies of $M'_{i-1}$ in $M'_i$ (numbered from left to right).
Therefore, we use the two functions $p(x)=(4\cdot x) \bmod 7$ and $q(x)=(4\cdot x + 1) \bmod 7$.
We define
\[
c(v,i)= \left\{\renewcommand{\arraystretch}{1.3}\begin{array}{ll}
c(v,i-1)    & \text{if } v  \text{ is from the first copy,}\\
p(c(v,i-1)) & \text{if } v  \text{ is from the second copy, and}\\
q(c(v,i-1)) & \text{if } v  \text{ is from the third copy.}\\
\end{array}\right.~\]

It remains to show that $c$ leads to an oriented coloring for $M_i$. Let $C_i=(W_i,F_i)$
with $W_i=\{0,1,2,3,4,5,6\}$ and $F_i=\{(c(u,i),c(v,i))\mid  (u,v)\in E_i\}$ be the color graph of $M'_i=(V_i,E_i)$.
By the definition of $M'_i$ we follow that
\[
\begin{array}{lclllll}
F_{i}&=&F_{i-1}&\cup &\{(p(x),p(y)) \mid (x,y)\in F_{i-1}\} \\
       & &    &\cup & \{(q(x),q(y)) \mid (x,y)\in F_{i-1}\} \\
       & &    &\cup & \{(p(c(v,i-1)),q(c(w,i-1))) \mid  v \text{ sink of } M'_{i-1}, w \text{ source of } M'_{i-1} \}.   \\

\end{array}
\]
In order to ensure an oriented coloring of  $M'_i$, we verify
that $C_i$ is an oriented graph. In Figure \ref{F09} the color graph
$C_i$ for $i\geq 5$ is given.

\begin{figure}[hbtp]
\centering
\centerline{\includegraphics[width=0.3\textwidth]{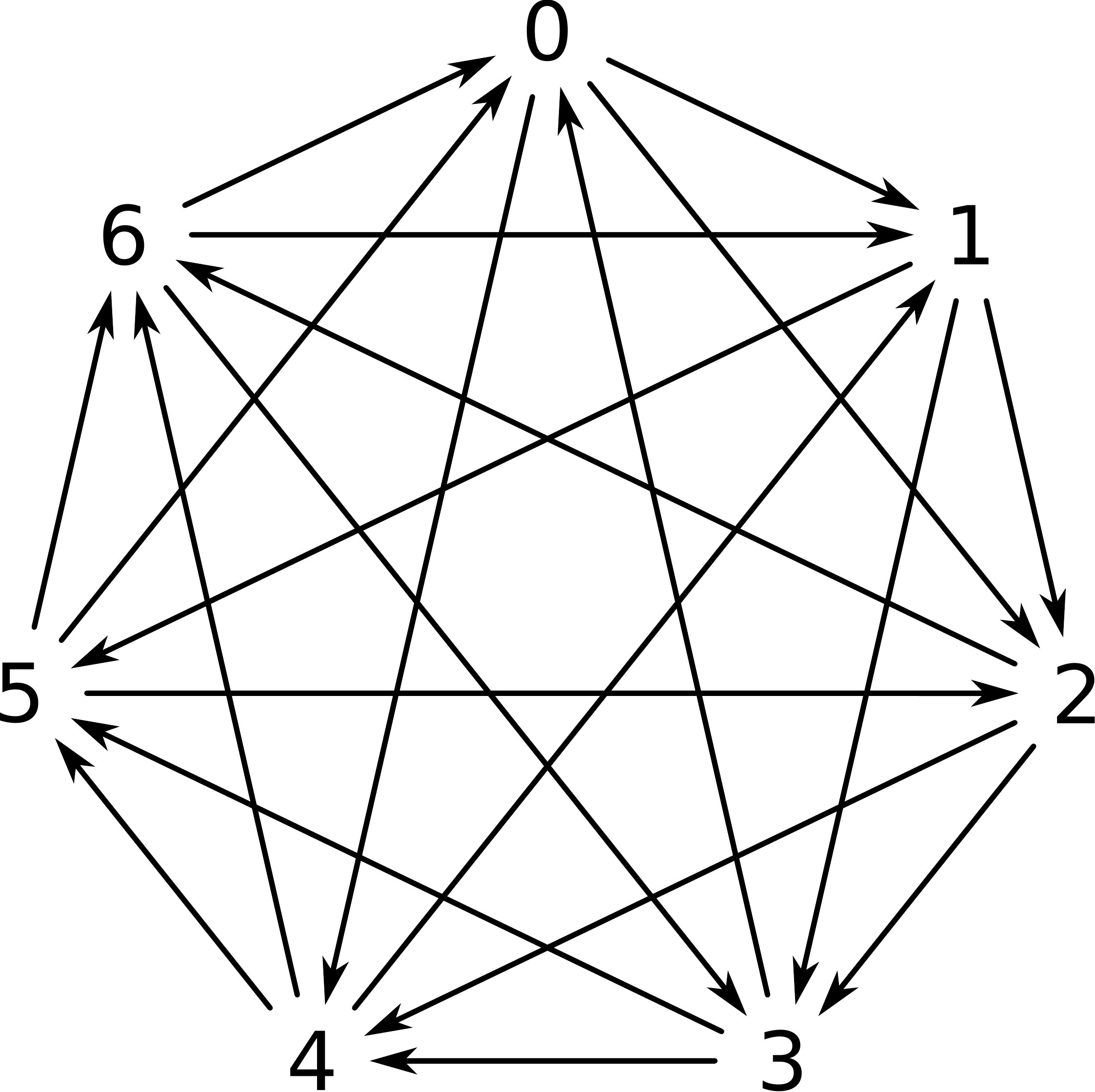}}
\caption{Color graph $C_i$ for $i\geq 5$ used in the proof of Theorem  \ref{algspd}, which 
is also known as the Paley tournament of order 7.}
\label{F09}
\end{figure}

Every source in $M'_i$ is colored by $0$ since $p(0)=0$.
Every sink in $M'_i$ is colored by $0$, $1$, or $5$ since  $q(0)=1$,
$q(1)=5$, and  $q(5)=0$.

Consequently, the arcs of $$\{(p(c(v,i-1)),q(c(w,i-1))) \mid  v \text{ sink of } M'_{i-1}, w \text{ source of } M'_{i-1}\}$$
belong to the set $$\{(p(0),q(0)),(p(1),q(0)),(p(5),q(0))\}=\{(0,1),(4,1),(6,1)\}.$$
For every $(u,v)\in\{(0,1),(4,1),(6,1)\}$ we know that
\begin{equation}
(v-u) \bmod 7 \in\{1,2,4\} \label{mod7}
\end{equation}
which implies $(u-v) \bmod 7 \not\in\{1,2,4\}$ and thus, condition (\ref{mod7}) does not hold
for the reverse arcs of $\{(0,1),(4,1),(6,1)\}$. It remains to show that
 (\ref{mod7}) remains true for all arcs $(u,v)$ when applying $p$ and $q$ to $M'_{i-1}$:
\[
\begin{array}{lclllll}
(q(v)-q(u)) \bmod 7&=& (((4\cdot v +1) \bmod 7) - ((4\cdot u +1) \bmod 7)) \bmod 7 \\
                  & = & (((4\cdot v)  \bmod 7) - ((4\cdot u)  \bmod 7)) \bmod 7\\
            & = & (p(v)-p(u)) \bmod 7\\
            & = &(4 (v-u))  \bmod 7
\end{array}
\]
Since $(v-u) \bmod 7 \in\{1,2,4\}$ leads to  $(4(v-u)) \bmod 7 \in\{1,2,4\}$, the result
follows. 
\end{proof}

Digraph $G$ on 27 vertices defined by expression $X$ in Example \ref{ex-six} satisfies $\chi_o(G)=7$, which was
found by a computer program.\footnote{We implemented
an algorithm which takes an oriented graph $G$ and an integer $k$ as an input and which decides
whether $\chi_o(G)\leq k$.} This implies
that the bound of Theorem \ref{algspd} is best possible.

\begin{example}\label{ex-six} In the following msp-expression we assume that
the  series composition $\times$ binds more strongly than the parallel composition $\cup$.
$$
\begin{array}{c}
X=v_{1} \times (v_2 \cup v_3 \times (v_4 \cup v_5 \times v_6))\times (v_7 \cup (v_8 \cup v_9 \times v_{10})\times (v_{11} \cup v_{12} \times v_{13}))\times \\
(v_{14}\cup (v_{15} \cup (v_{16} \cup v_{17} \times v_{18})\times (v_{19} \cup v_{20} \times v_{21}))\times (v_{22} \cup(v_{23} \cup v_{24} \times v_{25})\times v_{26}) ) \times v_{27}
\end{array}
$$
\end{example}


In oder to compute the  oriented
chromatic number of an msp-digraph $G$ defined by an msp-expression $X$,
we recursively compute the set $F(X)$ of all triples $(H,L,R)$ such that
$H$ is a color graph for $G$, where $L$ and $R$ are the sets of
colors of all sinks and all sources in $G$ with respect
to the coloring by $H$.
 The number of vertex labeled, i.e., the vertices are distinguishable
from each other, oriented graphs  on $n$
vertices is  $3^{\nicefrac{n(n-1)}{2}}$.
By Theorem \ref{algspd} we can conclude that
$$|F(X)|\leq 3^{\nicefrac{7(7-1)}{2}}\cdot 2^7 \cdot 2^7\in \bigo(1)$$
which is independent of the size of $G$.

For two color graphs $H_1=(V_1,E_1)$ and  $H_2=(V_2,E_2)$
we define $H_1+H_2=(V_1\cup V_2,E_1\cup E_2)$.

\begin{lemma}\label{le1}
\begin{enumerate}
\item\label{le1-1} 
For every $v\in V$ it holds  $F(v)=\{((\{i\},\emptyset ),\{i\},\{i\}) \mid 0\leq i \leq 6\}$.

\item\label{le1-2}  For every two msp-expressions $X_1$ and  $X_2$ we obtain set $F(X_1\cup X_2)$
from sets $F(X_1)$ and $F(X_2)$ as follows. For every $(H_1,L_1,R_1) \in F(X_1)$ and
every $(H_2,L_2,R_2) \in F(X_2)$ such that graph $H_1+H_2$ is oriented, we
put $(H_1 + H_2,L_1\cup L_2,R_1\cup R_2)$ into $F(X_1\cup X_2)$.

\item\label{le1-3}  For every two msp-expressions $X_1$ and  $X_2$ we obtain set $F(X_1\times X_2)$
from sets $F(X_1)$ and $F(X_2)$ as follows. For every $(H_1,L_1,R_1) \in F(X_1)$ and
every $(H_2,L_2,R_2) \in F(X_2)$ such that graph $H_1+H_2$ together with
the arcs in $R_1\times L_2$ is oriented, we
put $((V_1\cup V_2,E_1\cup E_2\cup R_1\times L_2),L_1,R_2)$  into $F(X_1\times X_2)$.
\end{enumerate}
\end{lemma}

\begin{proof}
\begin{enumerate}
  \item
  Set $F(v)$ includes obviously all possible solutions to color every vertex
on its own with the seven given colors.
  \item
  Let $(H_1,L_1,R_1)$ be any possible solution for coloring $\g(X_1)$,
which therefore is included in $F(X_1)$, as well as a possible solution
$(H_2,L_2,R_2)$ for coloring $\g(X_2)$ which is included in $F(X_2)$.
Let further $H_1+H_2$ be an oriented graph. Since the operation $\cup$
creates no additional edges in $\g(X_1\cup X_2)$, the vertices of
$\g(X_1)$ can still be colored with $H_1$ and the vertices of $\g(X_2)$
can still be colored with $H_2$ such that all vertices from $\g(X_1\cup
X_2)$ are legally colored. Further, all sinks in $\g(X_1)$ and $\g(X_2)$
are also sinks in $\g(X_1\cup X_2)$.
  The same holds for the sources. For an oriented digraph $H_1+H_2$ this
leads to $(H_1+H_2,L_1\cup L_2, R_1\cup R_2)\in F(X_1\cup X_2)$.

  Let $(H,L,R)\in F(X_1\cup X_2)$, then there is an induced subdigraph
$H_1$ of the color graph $H$ which colors $\g(X_1)$, an induced
subdigraph of $\g(X_1\cup X_2)$. Since $H$ is oriented, $H_1$ is oriented.
Let $L_1\subseteq L$ be the sources with vertices in $\g(X_1)$ and
$R_1\subseteq R$ be the sinks for vertices in $\g(X_1)$. Then, it holds
that $(H_1,L_1,R_1)\in F(X_1)$. The same arguments hold for $X_2$, such
that $(H_2,L_2,R_2)\in F(X_2)$.

  \item
  Let $(H_1,L_1,R_1)$ be any possible solution for coloring $\g(X_1)$,
which therefore is included in $F(X_1)$, as well as a possible solution
$(H_2,L_2,R_2)$ for coloring $X_2$ which is included in $F(X_2)$.
Further, let $H_1+H_2$ together with edges from $R_1\times L_2$ be an
oriented graph. Then, $H=(V_1\cup V_2,E_1\cup E_2\cup R_1\times L_2)$ is
an oriented coloring for $X=X_1\times X_2$. Since the sinks of $\g(X_1)$ are
connected with the sources of $\g(X_2)$
  in $\g(X)$ the sources of $L_1$ are the only sources left in $\g(X)$
as well as the sinks in $R_2$ are the only sinks left in $\g(X)$. This
leads to $(H,L_1,R_2)\in F(X)$.

  Let $(H,L,R)\in F(X_1\times X_2)$, then there is an induced subdigraph
$H_1$ of the color graph $H$ which colors $\g(X_1)$ which is an induced
subdigraph of $\g(X_1\times X_2)$. Since $H$ is oriented, $H_1$ is also
oriented. Since all the sources of $\g(X_1\times X_2)$ are in $\g(X_1)$
it holds that $L_1=L$ are also sources of $\g(X_1)$. Let $R_1$ be the
vertices in $\g(X_1)$ which only have out-going neighbors in $\g(X_2)$
but not in $\g(X_1)$,
  then $R_1$ are the sinks of $\g(X_1)$.
  Thus, it holds that $(H_1,L_1, R_1)\in F(X_1)$.
  Simultaneously, there is an induced subdigraph $H_2$ of the color graph
$H$ which colors $\g(X_2)$ which is an induced subdigraph of $\g(X_1\times
X_2)$. Since $H$ is oriented, $H_2$ is also oriented. Since all the
sinks of
  $\g(X_1\times X_2)$ are in $\g(X_2)$ it holds that $R_2=R$ are also
sinks of $\g(X_2)$. Let $L_2$ be the vertices in $\g(X_2)$ which only
have in-going neighbors in $\g(X_1)$ but not in $\g(X_2)$, then $L_2$
are the sources of $\g(X_2)$.
  Thus, it holds that $(H_2,L_2, R_2)\in F(X_2)$.
 \end{enumerate}
 This shows the statements of the lemma.
\end{proof}

Since every possible coloring of $G$ is part of the set $F(X)$, where $X$ is an msp-expression for 
$G$, it is possible to find a minimum coloring for $G$.

\begin{corollary}\label{cor1}
There is an oriented $r$-coloring for some msp-digraph $G$
which is given by some msp-expression $X$ if and only if there is some $(H,L,R)\in F(X)$
such that color graph $H$ has $r$ vertices.
Therefore, $\chi_o(G)=\min\{|V| \mid ((V,E),L,R)\in F(X)\}$.
\end{corollary}

\begin{theorem}\label{algspdd}
Let $G$ be an msp-digraph. Then, the oriented
chromatic number of $G$ can be computed in linear time.
\end{theorem}

\begin{proof}
Let $G$ be an msp-digraph on $n$ vertices and $m$ edges. Further, let  $T$ be an msp-tree for $G$
with root $r$. 
For some vertex $u$ of $T$ we denote by $T_u$
the subtree rooted at $u$ and $X_u$ the msp-expression defined by $T_u$.

In order to solve  the Oriented  Chromatic Number problem  for some msp-digraph  $G$, 
we traverse  msp-tree $T$ in a bottom-up order.
For every vertex $u$ of $T$ we compute $F(X_u)$
following the rules given in Lemma \ref{le1}. By Corollary~\ref{cor1} we can solve our
problem by $F(X_r)=F(X)$.

An msp-tree $T$ can be computed in $\bigo(n+m)$ time from
msp-digraph $G$,  see \cite{VTL82}.
Our rules given in Lemma~\ref{le1} show the following running times.
\begin{itemize}
\item
For every vertex $v\in V$ set  $F(v)$ is computable in $\bigo(1)$ time.

\item 
For every two msp-expressions $X_1$ and $X_2$
set $F(X_1  \cup X_2)$   can be computed in  $\bigo(1)$ time from $F(X_1)$ and $F(X_2)$.

\item 
For every two msp-expressions $X_1$ and $X_2$
set $F(X_1  \times X_2)$ can be computed
in $\bigo(1)$ time from $F(X_1)$ and $F(X_2)$.
\end{itemize}
Since we have $n$ leaves and $n-1$ inner vertices in msp-tree $T$,
the running time is in $\bigo(n+m)$.
\end{proof}

Due Corollary \ref{cor-tt-a} we know that
for every oriented  co-graph $G$  it holds that
$\chi_o(G)=\chi(\un(G))$. This equality does not hold for
 msp-digraphs by Example \ref{ex-six}
 and the next result.

\begin{proposition}\label{le-msp-uG}
Let $G$ be an msp-digraph. Then, it holds  that $\chi(\un(G))\leq 3$.
\end{proposition}

\begin{proof}
We show the result by giving a 3-coloring $c$ of $\un(M'_i)$ for the
oriented graphs $M'_i$ defined in Lemma \ref{le-mi}.
For some vertex $v$ of  $\un(M'_i)$ we define by $c(v,i)$ the color of $v$ in  $\un(M'_i)$.
First we color $\un(M'_0)$  by assigning color $0$ to the single
vertex in $\un(M'_0)$. For $i\geq 1$
we define the colors for the vertices $v$ in $\un(M'_i)=\un(M'_{i-1})\cup (\un(M'_{i-1} \times M'_{i-1}))$
according to the three copies of $\un(M'_{i-1})$ in $\un(M'_i)$ (numbered from left to right).
Therefore, we use the two functions $p(x)=(2\cdot x) \bmod 3$ and $q(x)=(2\cdot x +1) \bmod 3$. We
define
\[
c(v,i)= \left\{\renewcommand{\arraystretch}{1.3}\begin{array}{ll}
c(v,i-1)    & \text{if } v  \text{ is from the first copy,}\\
p(c(v,i-1)) & \text{if } v  \text{ is from the second copy, and}\\
q(c(v,i-1)) & \text{if } v  \text{ is from the third copy.}\\
\end{array}\right.
\]

In order to show the correctness of the given coloring for $\un(M'_i)$ by Definition \ref{def-oricol}
it suffices to verify that the end vertices of every edge in
$\un(M'_i)$ are colored differently. This follows since differently labeled
endvertices remain labeled differently when applying the permutations by
$p$ and $q$. Further, every  edge $\{u,v\}$  in  $\un(M'_i)$
arises from some arc  $(u,v)$ in $M'_i$  defined by the series composition and
satisfies $c(u,i)=0$ and  $c(v,i)=1$ or $c(u,i)=2$ and  $c(v,i)=1$. This holds true since
before applying the permutations by
$p$ and $q$ all sources $w$ are labeled by $c(w,i-1)=0$ and all sinks $v$ are labeled
by $c(v,i-1)=0$ or $c(v,i-1)=1$.
\end{proof}

For expression $X_1$ given in Example \ref{ex-msp} we
obtain by $\un(\g(X_1))$ a cycle on five vertices $C_5$ of
chromatic number $3$, which implies
that the bound of Proposition \ref{le-msp-uG} is sharp.
In \cite{GKL21} we introduced the concept of $g$-oriented $r$-colorings which 
generalizes both oriented colorings and colorings of the underlying 
undirected graph.

\section{Parameterized Results}

A {\em parameterized problem} is a pair $(\Pi,\kappa)$, where $\Pi$
is a decision problem, ${\mathcal I}$ the set of all instances
of $\Pi$ and $\kappa: \mathcal{I} \to \IN$
a so-called {\em parameterization}. 
The idea of parameterized algorithms is to restrict the combinatorial explosion to a parameter 
$\kappa(I)$ that is expected to be
small for all inputs $I\in  \mathcal{I}$.

An algorithm $A$ is an {\em FPT-algorithm with respect to $\kappa$}, if
there is a computable function $f: \IN \to \IN$ 
such that for every instance $I\in {\mathcal I}$ the running time
of $A$ on $I$ is at most $f(\kappa(I))\cdot |I|^{\bigo(1)}$  or 
equivalently at most $f(\kappa(I))+ |I|^{\bigo(1)}$. 
If there is an fpt-algorithm with respect to $\kappa$ that decides $\Pi$
then $\Pi$ is called {\em fixed-parameter tractable}.
FPT is the class of all parameterized problems which can
be solved by an FPT-algorithm.

An algorithm $A$ is an {\em XP-algorithm with respect to $\kappa$}, if
there are two computable functions $f,g: \IN \to \IN$ 
such that for every instance $I\in {\mathcal I}$ the running time
of $A$ on $I$ is at most $f(\kappa(I))\cdot |I|^{g(\kappa(I))}$.
If there is an xp-algorithm with respect to $\kappa$ which decides $\Pi$
then $\Pi$  is called {\em slicewise polynomial}.
XP is the class of all parameterized problems which can
be solved by an XP-algorithm.

In order to show fixed-parameter intractability, it is useful to 
show the hardness with respect  to one of the classes $\w[t]$ for some 
$t\geq 1$, which were introduced by  Downey and Fellows \cite{DF95} 
in terms of weighted satisfiability problems on classes of circuits. 
The relations $$\fpt \subseteq \w[1] \subseteq \w[2] \subseteq \ldots \subseteq
\xp$$ are called {\em $\w$-hierarchy} and 
all inclusions are assumed to be strict.


In the case of hardness with respect to some parameter $\kappa$ a natural
question is whether the problem remains hard for {\em combined} parameters, i.e.\
parameters $(\kappa_1,\ldots,\kappa_r)$ that consists of $r\geq 2$ parts of the input. 
The given notations can be carried over to combined parameters, e.g.\
an FPT-algorithm with respect to $(\kappa_1,\ldots,\kappa_r)$ is an algorithm of
running time $f(\kappa_1(I),\ldots,\kappa_r(I))\cdot |I|^{\bigo(1)}$ for
some computable function $f: \IN^r \to \IN$ depending only on $\kappa_1,\ldots,\kappa_r$.


For several problems the {\em standard parameter}, i.e.\ the threshold value given in the instance, 
is very useful.  Unfortunately, for the Oriented Chromatic Number problem the standard
parameter is the number of necessary colors and
does even not allow an $\xp$-algorithm, since $\OCN_{4}$ is NP-complete  \cite{KMG04}.

\begin{corollary}\label{cor-xp-r}
The  Oriented Chromatic Number problem  is not in $\xp$ when parameterized by $r$, unless  $\p=\np$.
\end{corollary}

\medskip
From an algorithmic point of view so called so-called {\em structural parameters}, which
are measuring the difficulty of decomposing a graph into a special tree-structure, are interesting.

For undirected graphs the clique-width \cite{CO00} and  tree-width \cite{RS86}
are the most important structural parameters. 
Clique-width is more general than  tree-width since
graphs of bounded tree-width have also bounded clique-width~\cite{CR05}. Conversely,
the tree-width can only be bounded by the clique-width under
certain conditions \cite{GW00}. A lot of
NP-hard graph problems admit poly\-nomial-time solutions when restricted to
graphs of bounded tree-width or graphs of bounded clique-width \cite{EGW01a}.

For directed graphs there are several attempts
to generalize tree-width  such as directed tree-width, directed path-width, DAG-width, or Kelly-width, which
are representative for what people are working on, see the surveys
\cite{GHKLOR14,GHKMORS16}. Unfortunately, none of these attempts
allows polynomial-time algorithms for a large class of problems on
digraphs of bounded width. This also holds for $\OCN_r$ and $\OCN$ by the following section.

\subsection{Parameterization by directed tree-width and related parameters}


As mentioned above, for $r\geq 4$ the $r$-Oriented Chromatic Number problem is hard on
DAGs \cite{CD06}. 
An $\xp$-algorithm with respect to  parameters directed tree-width, directed path-width, Kelly-width,
and DAG-width of the input digraph
would imply a polynomial time algorithm for every fixed parameter,
but even for parameter values $0$ or $1$ the problems are
NP-hard, since DAGs have width $0$ or $1$ for these parameters.

\begin{corollary}\label{cor-xp-ro}
The  Oriented Chromatic Number problem and for every positive integer $r\geq 4$ 
the $r$-Oriented Chromatic Number problem is not in $\xp$ when parameterized by directed 
tree-width, directed path-width, Kelly-width, or DAG-width, unless  $\p=\np$.
\end{corollary}

\subsection{Parameterization by  directed clique-width}

By  \cite{GHKLOR14}, directed
clique-width performs much better than directed path-width, directed tree-width, DAG-width, and
Kelly-width from the parameterized complexity point of view.
Hence, we consider the parameterized complexity of $\OCN$   parameterized by
directed clique-width. 
The directed clique-width of digraphs
has been defined by  Courcelle
and Olariu \cite{CO00} as follows.

\begin{definition}[Directed clique-width]\label{D4}
The {\em directed clique-width} of a digraph $G$, $\dcws(G)$ for short,
is the minimum number of labels needed to define $G$ using the following four operations:

\begin{enumerate}
\item Creation of a new vertex $v$ with label $a$ (denoted by $a(v)$).

\item Disjoint union of two labeled digraphs $G$ and $H$ (denoted by $G\oplus H$).

\item Inserting an arc from every vertex with label $a$ to every vertex with label $b$
($a\neq b$, denoted by $\alpha_{a,b}$).

\item Change label $a$ into label $b$ (denoted by $\rho_{a\to b}$).
\end{enumerate}
An expression $X$ built with the operations defined above 
using $k$ labels is called a {\em directed  clique-width $k$-expression}.
Let $\g(X)$ be the digraph defined by $k$-expression $X$.
\end{definition}

A class of graphs $\mathcal{L}$ has {\em bounded  directed clique-width} if there is some integer 
$k$ such that every graph in $\mathcal {L}$ has  directed clique-width at most $k$.

\begin{proposition}\label{prop-cw-oc}
Every oriented co-graph has  directed clique-width at most 2.
\end{proposition}

\begin{proof}
We  next show how to give a directed clique-width $2$-expression 
for some oriented co-graph $G$ given by some di-co-expression.
The vertices of $G$ are labeled by $a$ and within an 
order composition we relabel the vertices of one involved
subgraph to $b$.
The following three expressions recursively allow  to give a directed clique-width $2$-expression $X$ for $G$.
\begin{itemize}

\item If $G=(\{v\},\emptyset)$, then $X=a(v)$.

\item If $G=G_1 \oplus G_2$, such that $G_1$ and $G_2$ are defined by a directed clique-width $2$-expression $X_1$ 
and $X_2$, respectively, then we obtain a directed clique-width $2$-expression for $G$ by  
$X=X_1 \oplus X_2$.

\item If $G=G_1 \oslash G_2$, such that $G_1$ and $G_2$ are defined by a directed clique-width $2$-expression $X_1$ 
and $X_2$, respectively, then  we obtain a directed clique-width $2$-expression for $G$ by 
$$X= \rho_{b\to a}(\alpha_{a,b}(X_1 \oplus \rho_{a\to b}(X_2))).$$
\end{itemize}
This shows the statements of the proposition.
\end{proof}

In \cite{GKR21b} the set of oriented co-graphs is characterized by 
excluding cycles of length 2, i.e.\ $\overleftrightarrow{P_2}=(\{u,v\},\{(u,v),(v,u)\})$ as a proper subset
of the set of all graphs of directed clique-width 2, while for the  undirected versions
both classes are equal~\cite{CO00}.

\begin{proposition}\label{prop-cw-msp}
Every msp-digraph has directed clique-width at most $7$.
\end{proposition}

\begin{proof}
We next show how to give a directed clique-width $7$-expression 
for some msp-digraph $G$ given by some msp-expression. Therefore we
use the following four labels for the vertices of $G$:
\begin{itemize}
\item 
All vertices which are source and sink in $G$ are labeled by $a$.

\item 
All vertices which are sink (and no source) in $G$ are labeled by $b$.

\item 
All vertices which are source (and no sink) in $G$ are labeled by $c$.

\item
All vertices which are no sink and no source in $G$ are labeled by $d$.
\end{itemize}
Furthermore, we use three auxiliary labels $a',b',c'$ to distinguish the sinks and sources
of the two combined graphs within a series composition.

The following three expressions recursively allow  to give a directed clique-width $7$-expression $X$ for $G$.
\begin{itemize}

\item If $G=(\{v\},\emptyset)$, then $X=a(v)$.

\item If $G=G_1 \cup G_2$, such that $G_1$ and $G_2$ are defined by a directed clique-width $7$-expression $X_1$ 
and $X_2$, respectively, then we obtain a directed clique-width $7$-expression for $G$ by  
$X=X_1 \oplus X_2$.

\item If $G=G_1 \times G_2$, such that $G_1$ and $G_2$ are defined by a directed clique-width $7$-expression $X_1$ 
and $X_2$, respectively, then  we obtain a directed clique-width $7$-expression for $G$ by 
$$X= \rho''(\alpha'(X_1 \oplus \rho'(X_2))),$$
where $\rho'(X')=\rho_{c\to c'}(\rho_{b\to b'}(\rho_{a\to a'}(X')))$ relabels all sources and sinks in $\g(X')$ 
to the corresponding auxiliary labels, $\rho''(X')=\rho_{c'\to c}(\rho_{b'\to b}(\rho_{a'\to a}(X')))$
relabels all sources and sinks in $\g(X')$ 
to the corresponding original labels, and 
$$\alpha'(X'\oplus X'')=\alpha_{a,a'}(\alpha_{a,b'}(\alpha_{c,a'}(\alpha_{c,b'}(X' \oplus X''))))$$
inserts the edges of the series composition
between $\g(X')$ and $\g(X'')$.
\end{itemize}
This shows the statements of the proposition.
\end{proof}


By the given definition every graph of directed clique-width at most $k$ can be represented by a tree structure,
denoted as {\em $k$-expression-tree}. The leaves of the  $k$-expression-tree represent the
vertices of the digraph and the inner nodes of the  $k$-expression-tree  correspond
to the operations applied to the subexpressions defined by the subtrees.
Using the  $k$-expression-tree many hard problems have been shown to be
solvable in polynomial time when restricted to graphs of bounded directed clique-width \cite{GWY16,GHKLOR14}.

In order to show  fixed parameter tractability for  $\OCN_{r}$ w.r.t.\ the
parameter directed clique-width one can use its defineability within monadic second order logic (MSO).
%
%
We restrict to $\MSOA$-logic, which allows propositional logic,
variables for vertices and vertex sets of digraphs, the predicate $\arc(u,v)$ for arcs of digraphs, and
quantifications over vertices and vertex sets \cite{CE12}.
For defining optimization problems we use the $\LMSOA$  framework given in~\cite{CMR00}.

The following theorem is from \cite[Theorem 4.2]{GHKLOR14}.

\begin{theorem}[\cite{GHKLOR14}]\label{th-ghk}
For every integer $k$ and $\MSOA$  formula $\psi$, every $\psi$-$\LMSOA$  optimization problem
is fixed-parameter tractable on digraphs of clique-width $k$, with the parameters $k$ and $|\psi|$.
\end{theorem}

In \cite[Proposition 4.19]{GHKLOR14} the following  monadic second order logic formula  $\psi$ for $\OCN_{r}$ is given.

\begin{remark}Let $G=(V,E)$ be an oriented graph.
We can define $\OCN_{r}$ by the $\MSOA$ formula
$$\begin{array}{lcl}
\psi&=&\exists V_1,\ldots, V_r: \\
&&\displaystyle\left( \bigwedge_{i=1,\ldots,r} \forall x,y\in V_i \left(\neg \arc(x,y)\right) \wedge \bigwedge_{i,j=1,\ldots,r} \forall x,y \in X_i, z,t\in X_j \left(\arc(x,z) \rightarrow  \neg \arc(t,y)\right)     \right),
\end{array}
$$
where the vertices of set $V_i$ are mapped onto vertex $i$ of an orientation of a complete graph $K_r$.
\end{remark}

Since for the length of the given formula $\psi$ it holds  $|\psi|\in \bigo(r)$
by Theorem \ref{th-ghk} we obtain the following result.

\begin{corollary}\label{o}
The Oriented Chromatic Number problem
 is fixed parameter tractable on digraphs of clique-width $k$, with the parameters $k$ and $r$.
\end{corollary}

By Corollary \ref{o} we know the existence of an fpt-algorithm  for the Oriented Chromatic Number problem w.r.t.\ the
combined parameter of directed clique-width and standard parameter.
Next we give such an algorithm and estimate its running time.

\medskip
In order to compute the  oriented chromatic number of some oriented graph $G$ 
of oriented chromatic number at most $r$ which is defined by some 
directed $k$-clique-width expression $X$ we extend our solution given for msp-digraphs in Section \ref{sec-cog}.
Therefore, we recursively compute the set $F(X)$ of all labeled color graphs $H=(V,E)$ on $V=\{0,\ldots,r-1\}$. 
Every vertex $v\in V$ is labeled by  set $L\subseteq \{1,\ldots,k\}$
consisting of all clique-width labels of the vertices in $G$ with color $c$.
The number of vertex labeled, i.e., the vertices are distinguishable
from each other, oriented graphs  on $n$
vertices is  $3^{\nicefrac{n(n-1)}{2}}$. Since 
for every $v\in V$ there at most $2^{k}-1$ label sets $L$, we can conclude
$$|F(X)|\leq 3^{\nicefrac{r(r-1)}{2}} \cdot (2^{k}-1)^r  \leq 3^{\nicefrac{r^2}{2}}  \cdot 2^{k\cdot r} \leq 2^{r^2}  \cdot 2^{k\cdot r} = 2^{r(r+k)} $$
which is independent on the size of $G$.

For two color graphs $H_1=(V_1,E_1)$ and  $H_2=(V_2,E_2)$
we define $H_1+H_2=(V,E)$ as follows. Vertex set $V$ is obtained from
$V_1\cup V_2$ by merging the label sets of vertices for the same color.
Formally, for every $0\leq c \leq r-1$ we consider all $(c,L')\in V_1\cup V_2$
and replace them by $(c,\cup_{(c,L')\in  V_1\cup V_2}L' )$. Edge set $E$ is
obtained by $E_1\cup E_2$.

\begin{lemma}\label{le1cw}
\begin{enumerate}
\item For every $v\in V$ it holds  $F(a(v))=\{(\{(i,\{a\})\},\emptyset ) \mid 0\leq i \leq r-1\}$.

\item For every two $k$-expressions $X_1$ and  $X_2$ we obtain set $F(X_1\oplus X_2)$
from  sets $F(X_1)$ and $F(X_2)$ as follows. For every $H_1 \in F(X_1)$ and
every $H_2 \in F(X_2)$ such that graph $H_1+H_2$ is oriented we
put $H_1 + H_2$ into $F(X_1\oplus X_2)$.

\item Set $F(\alpha_{a,b}(X))$ can be obtained from  $F(X)$ as follows. First we remove from
$F(X)$ all color graphs $(V,E)$ such that there is some $(c_1,L_1)\in V$ and some $(c_2,L_2)\in V$
such that $((c_2,L_2),(c_1,L_1))\in E$ and $a\in L_1$ and $b\in L_2$. Afterwards
we modify every color graph $H=(V,E)$ in $F(X)$ as follows. If there is some $(c_1,L_1)\in V$ and some $(c_2,L_2)\in V$
such that $a\in L_1$ and $b\in L_2$ we insert $((c_1,L_1),(c_2,L_2))$ into $E$.
The resulting set is $F(\alpha_{a,b}(X))$.

\item It holds that $F(\rho_{a \to b}(X)) = \{\rho_{a \to b}((V,E)) \mid (V,E)\in F(X)\}$. Here we apply $\rho_{a \to b}((V,E))=(\rho_{a \to b}(V),\rho_{a \to b}(E))$,
$\rho_{a \to b}(V)=\{(c,\rho_{a \to b}(L))\mid(c,L)\in V\}$,  $\rho_{a \to b}(L)=\{\rho_{a \to b}(x)\mid x\in L \}$, $\rho_{a \to b}(x)=b$, if $x=a$ 
and $\rho_{a \to b}(x)=x$, if $x\neq a$,   $\rho_{a \to b}(E)=\{((c_1,\rho_{a \to b}(L_1)),(c_2,\rho_{a \to b}(L_2))) \mid ((c_1,L_1),(c_2,L_2))\}$.

\end{enumerate}
\end{lemma}

\begin{proof}
\begin{enumerate}
  \item Set $F(a(v))$ includes obviously all possible solutions to color every vertex
on its own with the $r$ given colors.
  \item Similar to (\ref{le1-2}.) of Lemma \ref{le1}.
  \item Let $F'(X)$ be the set obtained from  $F(X)$ using the modifications given in the lemma. We next show that $F'(X)=F(\alpha_{a,b}(X))$. Let $H=(V,E)\in F(X)$ be some color graph for $\g(X)$.  The modifications
  of $H$ given in the lemma either remove $H$ or lead to a  color graph for $\g(\alpha_{a,b}(X))$. Thus, it holds that $F'(X)\subseteq F(\alpha_{a,b}(X))$. Next, let  $H=(V,E)\in F(\alpha_{a,b}(X))$ be some color graph for $\g(\alpha_{a,b}(X))$. Then we can assume, that 
  there is some $(c_1,L_1)\in V$ and some $(c_2,L_2)\in V$
  such that $a\in L_1$ and $b\in L_2$ and $((c_1,L_1),(c_2,L_2))\in E$. Since the edge insertion $\alpha_{a,b}$
  does not change the vertices or clique-width labels there is 
  some color graph $H'=(V,E')\in F(X)$ for $\g(X)$ such that $H'$ is a subgraph of $H$ in which 
  the edge $((c_1,L_1),(c_2,L_2))$ can be missing.  The modifications
  of $H'$ given in the lemma insert the edge $((c_1,L_1),(c_2,L_2))$ into $H'$ which implies
  that $H$ is in  $F'(X)$.

   \item The relabeling $\alpha_{a,b}(X)$ can change the labels but not the structure of $\g(X)$.
   Thus, for every color graph $H=(V,E)$  in $F(X)$ for every vertex $(c,L)\in V$ we have to change 
   the set $L\subseteq \{1,\ldots,k\}$ of clique-width labels according to the performed relabeling
   operation $\rho_{a \to b}$, which is done by the rules given in the lemma. 
   
 \end{enumerate}
 This shows the statements of the lemma. 
\end{proof}

Since every possible coloring of $G$ is part of the set $F(X)$, where $X$ is a directed $k$-clique-width expression for 
$G$, it is possible to find a minimum coloring for $G$.

\begin{corollary} \label{cor3}
Let $G=(V,E)$ be an oriented graph given by a directed clique-width $k$-expression $X$.
There is an oriented $r$-coloring for $G$
if and only if there is some $H \in F(X)$ which has  $r$ vertices.
Therefore, $\chi_o(G)=\min\{|V| \mid (V,E)\in F(X)\}$.
\end{corollary}

\begin{theorem}\label{xp-dca}
The Oriented Chromatic Number problem on digraphs on $n$ vertices given by a directed
clique-width $k$-expression can be solved
in  $\bigo(n\cdot k^2 \cdot  2^{r(r+k)})$ time.
\end{theorem}

\begin{proof}
Let $G=(V,E)$  be a digraph of directed clique-width at most $k$ and $T$ be a $k$-expression-tree
for $G$ with root $w$.
For some vertex $u$ of $T$ we denote by $T_u$
the subtree rooted at $u$ and $X_u$ the $k$-expression defined by $T_u$.
In order to solve the Oriented Chromatic Number problem for $G$,
we traverse $k$-expression-tree  $T$ in a bottom-up order.
For every vertex $u$ of $T$ we compute $F(X_u)$ following the rules
given in Lemma \ref{le1cw}. By Corollary \ref{cor3} we can solve our
problem by $F(X_w)=F(X)$.

Our rules given Lemma \ref{le1cw} show the following running times.
For every $v\in V$ and  $a\in\{1,\ldots,k\}$ set $F(a(v))$
can be computed in $\bigo(1)$.
The set $F(X \oplus Y)$ can be computed
in time $\bigo(2^{r(r+k)})$  from $F(X)$ and $F(Y)$.
The sets $F(\alpha_{a,b}(X))$ and 
$F(\rho_{a \to b}(X))$  can be computed
in time    $\bigo(2^{r(r+k)})$      from $F(X)$.

In order to bound the number and order of operations within directed clique-width expressions,
we can use the normal form for clique-width expressions defined in  \cite{EGW03}.
The proof of Theorem 4.2 in \cite{EGW03} shows that also for directed  clique-width expression $X$,
we can assume that for every subexpression, after a disjoint union operation
first there is a sequence of edge insertion operations followed by a sequence of
relabeling operations, i.e. between two disjoint union operations there is no relabeling before
an edge insertion. Since there are $n$ leaves in $T$, we have $n-1$
disjoint union operations, at most $(n-1)\cdot (k-1)$ relabeling operations,
and at most $(n-1)\cdot \nicefrac{k(k-1)}{2}$ edge insertion  operations.
This leads to an overall running time of $\bigo(n\cdot k^2 \cdot  2^{r(r+k)})$.
\end{proof}

%
%

%

Up to now there are only very few digraph classes for which we can compute a directed clique-width expression in
polynomial time. This holds for directed co-graphs, digraphs of bounded directed modular width, 
orientations of trees, and directed cactus forests.
For such classes we can apply the result of Theorem \ref{xp-dca}.
In order to find directed clique-width expressions for general
digraphs one can use results on the related parameter bi-rank-width \cite{KR13}.
By \cite[Lemma 9.9.12]{BG18} we can use approximate directed clique-width
expressions obtained from rank-decomposition with the
drawback of a  single-exponential blow-up on the parameter.

If we restrict to some constant value of $r$ the running time of Theorem \ref{xp-dca}
leads to the following result, which reproves a result shown in \cite{GHKLOR14} using 
Theorem \ref{th-ghk}.

\begin{corollary}\label{cor11ax}
For  every integer $r$ the 
$r$-Oriented Chromatic Number problem  is in $\fpt$ when parameterized by directed clique-width.
\end{corollary}

\begin{corollary}\label{ocnk2bb}
For every class of graphs of bounded  directed clique-width and every positive integer $r$ the $r$-Oriented Chromatic Number problem can be solved in polynomial time.
\end{corollary}

By Propositions \ref{prop-cw-oc} and \ref{prop-cw-msp} we know that on oriented co-graphs and 
msp-digraphs  for every positive integer $r$ the $r$-Oriented Chromatic Number problem can be solved in polynomial time. Comparing this to Corollary \ref{c-co} and Theorem \ref{algspdd} we have shown more
general solutions before.

If we restrict to some constant value of $k$ the running time of Theorem \ref{xp-dca}
leads to the following result.

\begin{corollary}\label{cor11ay}
For every class of graphs of bounded directed clique-width the 
Oriented Chromatic Number problem  is in $\fpt$ when parameterized by $r$.
\end{corollary}

%

\subsection{Parameterization by  directed modular-width}

Next, we discuss parameterization of $\OCN$ and $\OCN_r$ w.r.t.
the parameter directed modular-width ($\dmws$). Directed modular-width
was introduced and applied to the Acyclic Chromatic Number
problem in \cite{SW19,SW20}.

A class of graphs $\mathcal{L}$ has {\em bounded  directed modular-width} 
if there is some integer $k$ such that every graph in $\mathcal {L}$
has  directed modular-width at most $k$.


\begin{proposition}[\cite{SW20}]\label{prop-co-mw}
Every (directed and thus every) oriented co-graph has directed modular width $2$.
\end{proposition}

The relation of directed clique-width and directed modular width is as follows.

\begin{lemma}[\cite{SW20}]\label{le-mw-cw} For every digraph $G$ it holds that
$\dcws(G)\leq \dmws(G)$.
\end{lemma}

On the other hand, there exist several classes of digraphs of bounded directed clique-width and
unbounded directed modular width, e.g. the set of all oriented
paths $\{\overrightarrow{P_n} \mid n\geq 1\}$,
the set of all oriented cycles $\{\overrightarrow{C_n} \mid n\geq 1\}$, and the set of all msp-digraphs.
An advantage of directed modular width is that it can be computed efficiently  \cite{SW19,SW20}.

Lemma \ref{le-mw-cw} and Corollary \ref{cor11ax} lead to the following result.

\begin{corollary}\label{ocnk2}
For every positive integer $r$ the $r$-Oriented Chromatic Number problem
is fixed parameter tractable
w.r.t.\ the parameter directed modular-width.
\end{corollary}

\begin{corollary}\label{ocnk2a}
For every class of graphs of bounded  directed modular-width and every positive integer $r$ the $r$-Oriented Chromatic Number problem can be solved in polynomial time.
\end{corollary}

By Proposition \ref{prop-co-mw} we conclude that on oriented co-graphs 
for every positive integer $r$ the $r$-Oriented Chromatic Number problem can be solved in polynomial time. Comparing this to Corollary \ref{c-co}  we have shown a more
general solution before.


\subsection{Parameterization by the width of the underlying undirected graph}

In \cite{Gan09}, Ganian has shown  an $\fpt$-algorithm for $\OCN$ w.r.t.\ the parameter
tree-width (of the underlying undirected graph).
Further, he has shown that $\OCN$ is DET-hard for classes of
oriented graphs, such that the underlying undirected class has bounded rank-width.

\subsection{Parameterization by  number of vertices}

A positive result can be obtained for parameter ''number of vertices'' $n$.
Since integer linear programming is fixed-parameter tractable for the
parameter ''number of variables''  \cite{Len83}
the following
binary integer program for $\OCN$ using $n^2+n$ variables
implies an $\fpt$-algorithm for parameter $n$.

\begin{remark}\label{rem-bip}
To formulate   the  Oriented Chromatic Number problem for some oriented graph $G=(V,E)$ as a binary integer program,
we introduce a binary variable $y_j\in\{0,1\}$, $j\in \{1,\ldots, n\}$, such that
$y_j=1$ if and only if color $j$ is used. Further we use $n^2$ variables
$x_{i,j}\in\{0,1\}$, $i,j\in \{1,\ldots,n\}$, such that
$x_{i,j}=1$ if and only if vertex $v_i$ receives color  $j$. The main
idea is to ensure the two conditions of Definition \ref{def-oc}
within (\ref{01-p3-cn0x}) and (\ref{01-px}).

\begin{eqnarray}
\text{minimize}   \sum_{i=1}^{n}   y_i  \label{01-p1-cn0a}
\end{eqnarray}
subject to
\begin{eqnarray}
\sum_{j=1}^{n} x_{i,j}  & =  &1 \text{ for every }    i \in  \{1,\ldots,n\} \label{01-p2-cn0a} \\
x_{i,j} + x_{i',j}& \leq   &y_j \text{ for every }   (v_i,v_{i'})\in E, j \in \{1,\ldots,n\}  \label{01-p3-cn0x} \\
   \bigvee_{j=1}^n x_{i,j} \wedge x_{i''',j}             & \leq & 1- \bigvee_{j=1}^n x_{i',j} \wedge x_{i'',j}                            \text{ for every }   (v_i,v_{i'}),(v_{i''},v_{i'''})\in E \label{01-px}   \\
y_j   &\in &  \{0,1\} \text{ for every } j \in   \{1,\ldots,n\} \label{01-p4-cna0a} \\
x_{i.j}   &\in &  \{0,1\} \text{ for every } i,j \in  \{1,\ldots,n\}  \label{01-p5-cn0a}
\end{eqnarray}

Equations in (\ref{01-px})  are not in propositional logic. In order to reformulate
them for binary integer programming, one can use the results of \cite{Gur14}.
\end{remark}

\begin{corollary}\label{xp-n}
The  Oriented Chromatic Number problem  is in $\fpt$ when parameterized by 
the number of vertices $n$.
\end{corollary}

\section{Conclusions and outlook}

In this paper we considered oriented colorings of recursively defined digraphs.
We used the concept of perfect orderable graphs in order to show that for
acyclic transitive digraphs every greedy coloring along
a topological ordering leads to an optimal oriented coloring,
which generalizes a known dynamic programming solution for the Oriented
Chromatic Number problem on oriented co-graphs in \cite{GKR19d}.

Further, we showed that every msp-digraph
has oriented chromatic number at most $7$, which is
best possible. We applied this bound together with the recursive  structure of
msp-digraphs to give a linear time solution for
computing the oriented chromatic number of msp-digraphs.

In Figure \ref{grcl} we summarize the relation of special graph classes
considered in this work. Among these are directed acyclic graphs (DAGs),
transitive directed acyclic graphs (transitive DAGs), series-parallel digraphs ($\SPD$),
minimal series-parallel digraphs ($\MSP$), series-parallel order graphs ($\SPO$), and
oriented co-graphs ($\OC$).
The directed edges represent the existing relations between
the graph classes, which follow by their definitions.
For the relations to further graph classes we refer to  \cite[Figure 11.1]{BG18}.

\begin{figure}[ht]
\begin{center}
\centerline{\includegraphics[width=0.37\textwidth]{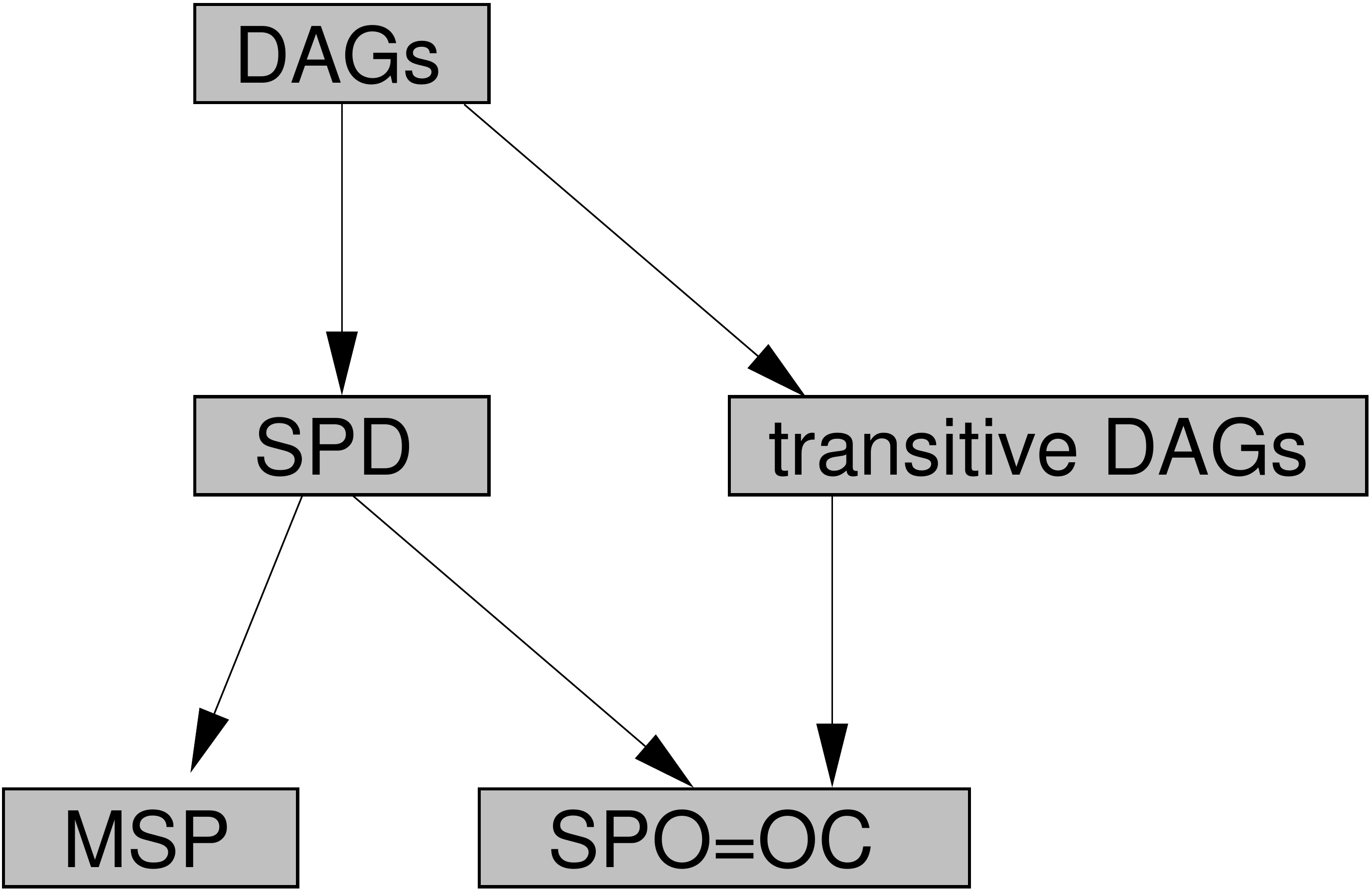}}
\caption{The figure shows the inclusions of special graph classes. A directed edge
from class $A$ to class $B$ indicates that $B\subseteq A$. Two classes $A$ and $B$
are incomparable if there is neither a directed path from $A$ to $B$, nor
a directed path from $B$ to $A$.}\label{grcl}
\end{center}
\end{figure}

Furthermore we considered the parameterized complexity of the Oriented Chromatic Number problem
by structural parameters, which
are measuring the difficulty of decomposing a graph into a special tree-structure.
The parameterized results of Corollary \ref{o} and Corollary \ref{cor11ax}
also hold for any parameter which is
larger or equal than directed clique-width
directed linear clique-width \cite{GR19c}. Furthermore, restricted to
semicomplete digraphs the shown parameterized solutions also hold
for directed path-width \cite[Lemma 2.14]{FP19}.

For future work it could be interesting to extend our solutions to further classes
such as edge series-parallel digraphs \cite{VTL82}. 
The parameterized complexity of $\OCN$ w.r.t.\ the parameter
directed clique-width remains open.
Since the directed clique-width of a digraph is always greater or equal the undirected
clique-width of the corresponding underlying undirected graph \cite{GWY16}, the result
of Ganian \cite{Gan09} does not imply a hardness result.
Furthermore, the parameterized complexity of
$\OCN$  w.r.t.\ structural parameter directed modular-width  
and of $\OCN_r$  w.r.t.\ parameter rank-width of $\un(G)$
remains open (cf.\ Table \ref{fpt-sum}).

\section*{Acknowledgements} \label{sec-a}

The work of the second and third author was supported
by the Deutsche
Forschungsgemeinschaft (DFG, German Research Foundation) -- 388221852




\newcommand{\etalchar}[1]{$^{#1}$}

\end{document}